\documentclass[a4paper,10pt]{article}
\usepackage{amssymb}
\usepackage{amsmath}

\newtheorem{theorem}{Theorem}[section]
\newtheorem{corollary}[theorem]{Corollary}
\newtheorem{proposition}[theorem]{Proposition}
\newtheorem{lemma}[theorem]{Lemma}
\newtheorem{example}[theorem]{Example}

\newcommand{\qed}{
  \ifmmode
   \eqno{\qedsymbol}
  \else
    \leavevmode\unskip\penalty9999 \hbox{}\nobreak\hfill\hbox{\qedsymbol}
  \fi
}
\newcommand{\qedsymbol}{\leavevmode\vrule height 1.2ex width 1.1ex depth -.1ex}
\newenvironment{proof}{\begin{trivlist}\item[\hskip
\labelsep{\bf Proof.\quad}]}
{\hfill\qed\rm\end{trivlist}}

\usepackage{tikz}
\usetikzlibrary{decorations.markings}
  
\let\cal=\mathcal

\mathchardef\emptyset="001F
\font\Bbb=msbm10 at 12 truept

\def\z{\hbox{\Bbb Z}}
\def\obj{\hbox{\rm Obj~}}
\def\mor{\hbox{\rm Mor}}
\def\leh{\le}
\def\geh{\ge}
\def\lem{\le_{\cal M}}
\def\gem{\ge_{\cal M}}
\def\omegah{\omega}
\def\omegam{\omega_{\cal M}}
\def\dom{{\text{dom}}}

\title{Actions of $E-$dense semigroups and an application to the discrete log problem}
\author{James Renshaw\\
\small Mathematical Sciences\\
\small University of Southampton\\
\small Southampton, SO17 1BJ\\
\small England\\
\small j.h.renshaw@maths.soton.ac.uk}

\begin{document}
\maketitle

\begin{abstract}
\noindent We describe the structure of $E-$dense acts over $E-$dense semigroups in an analogous way to that for inverse semigroup acts over inverse semigroups. This is based, to a large extent, on the work of Schein on representations of inverse semigroups by partial one-to-one maps. The motivation for this is to consider an application to the discrete log problem in cryptography.
\end{abstract}



\section{Introduction and Preliminaries}
Let $S$ be a semigroup. By a {\em left $S-$act} we mean a non-empty set $X$ together with an action $S\times X\to X$ given by $(s,x)\mapsto sx$ such that for all $x\in X, s,t\in S, (st)x = s(tx)$. If $S$ is a monoid with identity $1$, then we normally require that $1x=x$ for all $x\in X$. A {\em right $S-$act} is defined dually. If $X$ is a left $S-$act then the semigroup morphism $\rho:S\to {\cal T}(x)$ given by $\rho(s)(x) = sx$ is a representation of $S$. Here ${\cal T}(X)$ is the {\em full transformation} semigroup on $X$ consisting of all maps $X\to X$. Conversely, any such representation gives rise to an action of $S$ on $X$. If $X$ is both a left $S-$act and a right $T-$act for semigroups/monoids $S$ and $T$ and if in addition $(sx)t=s(xt)$ then $X$ is said to be an {\em $(S,T)-$biact}.  Throughout this paper, unless otherwise stated, all acts will be left $S-$acts. We refer the reader to~\cite{howie-95} for basic results and terminology in semigroups and monoids and to~\cite{ahsan-08} and~\cite{kilp-00} for those concerning acts over monoids. If $S$ is an inverse semigroup then we can replace ${\cal T}(X)$ by ${\cal I}(X)$, the inverse semigroup of partial one-to-one maps. A comprehensive theory of these types of representations was given by Boris Schein in the early 1960's and an account of that work can be found in~\cite{clifford} and~\cite{howie-95}. Here we wish to emulate that approach for {\em $E-$dense semigroups} and do so in section 2. In section 3, we apply some of these results to the discrete log problem found in cryptography (see for example~\cite{maze-07}).
\medskip

Recall that an {\em idempotent} in a semigroup $S$ is an element $s\in S$ such that $s^2=s$. A {\em band} is a semigroup consisting entirely of idempotents whilst a {\em semilattice} is a commutative band. We shall denote the idempotents of a semigroup $S$ as $E(S)$ or more generally $E$. Let $S$ be a semigroup and let $W(s)=\{s'\in S|s'ss'=s'\}$ be the set of {\em weak inverses} of $s$ and $V(s)=\{s'\in W(s)|s\in W(s')\}$ be the set of {\em inverses} of $s$. If $S$ is a group then clearly $W(s)=V(s)=\{s^{-1}\}$ for all $s\in S$, whilst if $S$ is a rectangular band, that is to say a band in which $xyx=x$ for all $x,y\in S$, then $W(s) = V(s)=S$ for each $s\in S$. Notice that if $s'\in W(s)$ then $s's, ss' \in E$. Moreover, if $e\in E$ then $e\in W(e)$. It may of course be the case that for a given element $s\in S, W(s)=\emptyset$. We do however have

\begin{lemma}[{\cite[Corollary 3.3]{weipoltshammer-04}}]\label{weak-inverse-product-lemma}
Let $S$ be a semigroup in which $E\ne\emptyset$. Then $E$ is a band if and only if for all $s,t\in S, W(st) = W(t)W(s)$.
\end{lemma}

From the proof of~\cite[Lemma 7.14]{fountain-04} we can deduce
\begin{lemma}\label{weakly-self-conjugate-lemma}
Let $S$ be a semigroup with band of idempotents $E$. Then for all $s\in S, s'\in W(s), e\in E$ it follows that $ses',s'es\in E$.
\end{lemma}
If the conclusions of Lemma~\ref{weakly-self-conjugate-lemma} hold, we say that $S$ is {\em weakly self conjugate}. We shall make frequent use of both the previous properties of semigroups in which $E$ is a band without further reference. Notice also that if $s'\in W(s)$ then $ss's\in V(s')\subseteq W(s')$, a fact that we shall also use frequently. In particular, $s'$ is a  regular element of $S$.

\medskip

It was shown by Mitsch~\cite{mitsch-86} that the following is a natural partial order on any semigroup $S$ 
$$
a\lem b \text{ if and only if there exists }x,y \in S, a=xb=by, xa=ay=a.
$$
Notice  that if there exist idempotents $e$ and $f$ such that $a=eb=bf$ then it follows that $a=ea=af$ and so $a\lem b$. If $a$ is a regular element of $S$ then it is easy to check that $e=aa'x\in E$ and $f=ya'a\in E$ for any $a'\in V(a)$, and that $a=eb=bf$. Hence if $a$ is regular then
$$
a\lem b \text{ if and only if there exists }e,f \in E, a=eb=bf.
$$
In particular, this is true if $a\in E$. It is also worth noting here that if $E$ is a semilattice then the restriction of $\lem$ to $E$ is compatible with multiplication, a fact that we shall use later.

\smallskip
Let $A$ be a subset of a semigroup $S$ and define
$$
A\omegam = \{s\in S|a\lem s\text{ for some }a\in A\}.
$$
If $A=\{a\}$ then we will write $A\omegam$ as $a\omegam$. Notice that $(A\omegam)\omegam = A\omegam$ and that if $A\subseteq B$ then $A\omegam\subseteq B\omegam$. Also, if $A\subseteq B\omegam$ then $A\omegam\subseteq B\omegam$. We call $A\omegam$ the {\em ($\omegam-$)closure} of $A$ and say that $A$ is ($\omegam-$)closed if $A = A\omegam$. Notice that if $A\subseteq B$ with $B$ being $(\omegam-)$closed, then $A\omegam\subseteq B$.

\bigskip

If $T$ is a subset of a semigroup $S$ then we say that $T$ is {\em left (resp. right) dense} in $S$ if for all $s\in S$ there exists $s'\in S$ such that $s's\in T$ (resp. $ss'\in T$). We say that $T$ is {\em dense} in $S$ if it is both left and right dense in $S$. We are particularly interested in the case where $T=E$ the set of idempotents of $S$ and we shall refer to semigroups in which $E$ is dense in $S$ as {\em $E-$dense} or {\em  $E-$inversive} semigroups. This concept was originally studied by Thierrin~\cite{thierrin-55} and subsequently by a large number of authors (see Mitsch~\cite{mitsch-99} for a useful survey article, but note that the term $E-$dense has a slightly different meaning there). Included in this class of semigroup are the classes of all regular semigroups, inverse semigroups, groups, eventually regular semigroups (that is to say every element has a power that is regular), periodic semigroups (every element is of finite order) and indeed finite semigroups.

\smallskip
Let $S$ be a semigroup, let $L(s) = \{s'\in S|s's\in E\}$. Then it is well known that $W(s)\subseteq L(s)$. Moreover, for each $s'\in L(s), s'ss'\in W(s)$ and so $W(s)\ne\emptyset$ if and only if $L(s)\ne\emptyset$. 
The following is then immediate.
\begin{lemma}
Let $S$ be an $E-$dense semigroup. Then for all $s\in S$ there exists $s'\in S$ such that $s's,ss' \in E$.
\end{lemma}

\smallskip

Let $S$ be an $E-$dense semigroup with a band of idempotents $E$ and define a partial order on $S$ by
$$
s\leh t\text{ if and only if either }s=t\text{ or there exists }e,f\in E\text{ with }s=te=ft
$$
and note that $\leh\;\subseteq\;\lem$ and that if $E$ is a semilattice then $\leh$ is compatible with multiplication by weak inverses. If $s$ is regular (in particular idempotent) then $s\leh t$ if and only if $s\lem t$. If $A$ is a subset of $S$ then define
\begin{gather*}
A\omegah = \{s\in S| a\leh s\text{ for some }a\in A\}
\end{gather*}
and notice that $A\subseteq A\omegah \subseteq A\omegam$. It is also clear that $(A\omegah)\omegah = A\omegah$. Note from above that if $A\subseteq E$ then $A\omegam = A\omegah$. We shall make use of $\omegah$ in Section 2.

\smallskip

Weak inverses of elements will not in general be unique and in section 2 we will often need to deal with more than one weak inverse of a given element. The following useful result is easy to establish.
\begin{lemma}\label{weak-inverses-lemma}
Let $S$ be an $E-$dense semigroup with a semilattice of idempotents $E$. 
\begin{enumerate}
\item If $s'\in W(s)$ and $e,f\in E$ then $es'f\in W(s)$.
\item $(W(s),\le)$ is a lower semilattice. In more detail, if $s', s^\ast\in W(s)$ then $s'ss^\ast\in W(s)$ and $s'ss^\ast = s'\wedge s^\ast$, the meet of $s'$ and $s^\ast$.
\item If $s'\in W(s)$ and $s'^\ast\in W(s')$ then $s'^\ast=s'^\ast s's = ss's'^\ast$ and for all $e\in E, es'^\ast\leh s$. In particular $s'^\ast\leh s$.
\item If $s'\in W(s)$ then $W(s') = sW(s)s$ and $V(s') = \{ss's\}$. In particular if $s^\ast\in W(s)$ then $W(s')=W(s^\ast)$ and $ss'ss^\ast s \in W(s')$.
\item Let $W = \{s'\in W(s)|s\in S\}$. Then $W$ is an inverse subsemigroup of $S$.
\item For all $s \in S,\; W(W(W(s)))=W(s)$.
\end{enumerate}
\end{lemma}
\begin{proof} Let $E$ and $S$ be as stated.
\begin{enumerate}
\item Suppose that $e,f \in E$ and $s'\in W(s)$. Then $(es'f)s(es'f) = es'fss'f = es'f$ and so $es'f\in W(s)$
\item This is straightforward on noting that $(s'ss^\ast)s(s'ss^\ast) = s'ss'ss^\ast ss^\ast = s'ss^\ast$.
Note that $s'ss^\ast = s^\ast ss'$. It is clear that $s'ss^\ast \le s', s^\ast$, so suppose that $t\le s',s^\ast$. Then there exist $e,f,g,h\in E(S)$ such that $t=es'=s'f=gs^\ast=s^\ast h$. Hence since $t=s'st = (s'ss^\ast)h$ and $t=tss' = gs^\ast ss' = g(s'ss^\ast)$ then $t\le s'ss^\ast$ as required.

\item\label{part3} If $s'^\ast\in W(s')$ then $s'^\ast=s'^\ast s's'^\ast = s'^\ast s'ss's'^\ast = ss's'^\ast s's'^\ast = ss's'^\ast = s'^\ast s's'^\ast s's = s'^\ast s's$. Finally notice that for all $e\in E$,
$$
s(s's'^\ast s'es'^\ast s's) = (ss's'^\ast s'es'^\ast s')s = s'^\ast s'es'^\ast = es'^\ast s's = es'^\ast 
$$
and so $es'^\ast \leh s$.
\item\label{part4} Clearly $sW(s)s\subseteq W(s')$. Let $s'^\ast\in W(s')$ so that by part~(\ref{part3}), $s'^\ast=ss's'^\ast = ss's'^\ast s's$ and  $s's'^\ast s'\in W(s)$.

Now let $s'^\ast\in V(s')\subseteq W(s')$. Since $s's'^\ast = s's'^\ast s's = s's$ then $s'^\ast = ss's'^\ast = ss's$.

If $s^\ast\in W(s)$ then $W(s^\ast) = sW(s)s = W(s')$ and since $s'ss^\ast \in W(s)$ then $ss'ss^\ast s\in W(s')$.
\item Clearly $W\ne\emptyset$. Let $s',t'\in W$ with $s'\in W(s), t'\in W(t)$ then it is easy to check that $s't'\in W(ts)\subseteq W$ and by part (\ref{part4}), $|V(s')| = 1$ for each $s'\in W$ and so $W$ is an inverse subsemigroup of $S$. Alternatively, note that $W={\hbox{ Reg}}(S)$ the set of regular elements of $S$.
\item Let $s'\in W(s), s'^\ast\in W(s'), {s'^\ast}'\in W(s'^\ast)$. Then by part (\ref{part3}), ${s'^\ast}'ss'^{\ast'} ={s'^\ast}'ss's'^\ast {s'^\ast}' = {s'^\ast}'s'^\ast {s'^\ast}'={s'^\ast}'$. On the other hand $s' = s'ss' = s'(ss's)s' \in W(W(s'))$ and the result follows easily.
\end{enumerate}
\end{proof}

\smallskip

\begin{lemma}[{\cite[Proposition 2]{mitsch-90}}]\label{e-dense-group-lemma}
A semigroup $S$ is a group if and only if for every element $s\in S$, $|L(s)|=1$.
\end{lemma}

It is also easy to see that
\begin{lemma}\label{e-dense-group-two-lemma}
Let $S$ be an $E-$dense monoid. Then $S$ is a group if and only if $|E|=1$.
\end{lemma}

A subset $A$ of a semigroup $S$ is called {\em unitary} in $S$ if whenever $sa\in A$ or $as\in A$ it necessarily follows that $s\in A$. If $E$ is a unitary subset of $S$ then we shall refer to $S$ as an {\em $E-$unitary semigroup}.

\begin{lemma}[{\cite{reither-94},~\cite[Theorem 6.8]{mitsch-99}}]\label{e-dense-e-unitary-lemma}
Let $S$ be an $E-$dense semigroup. Then $S$ is $E-$unitary if and only if $E$ is a band and $E\omegah=E$.
\end{lemma}

\begin{lemma}[{\cite[Proposition 1.2]{almeida-92}}]
Let $S$ be an $E-$unitary semigroup. For all $s \in S$, if $s'\in L(s)$ then $s\in L(s')$.
\end{lemma}

\section{$E-$dense actions of $E-$dense semigroups}
In this section we take inspiration from the theory of inverse semigroup actions, which in turn is based on Schein's representation theory of inverse semigroups by partial one-to-one maps (see~\cite{clifford} and~\cite{howie-95}).

\medskip

Let $S$ be an $E-$dense semigroup, let $X$ be a non-empty set and let $\phi:S\times X \to X$ be a partial map with the property that $\phi(st,x)$ exists if and only if $\phi(s,\phi(t,x))$ exists and then
$$
\phi(st,x) = \phi(s,\phi(t,x)).
$$
We will, as is usual, denote $\phi(s,x)$ as $sx$ and simply write $(st)x$ as $stx$ when appropriate. By a partial map we of course mean that not every element of $S$ need act on every element of $X$. A more formal definition can be found in~\cite{howie-95}. We say that $\phi$ is an {\em $E-$dense action} of $S$ on $X$, and refer to $X$ as an {\em $E-$dense $S-$act}, if
\begin{enumerate}
\item the action is {\em cancellative}; meaning that whenever $sx=sy$ then $x=y$;
\item the action is {\em reflexive}; that is to say, for each $s\in S$, if $sx$ exists then there exists $s'\in W(s)$ such that $s'(sx)$ exists.
\end{enumerate}
The {\em domain} of an element $s \in S$ is the set
$$
D_s^X = \{x \in X | sx\text{ exists}\}.
$$
We shall denote $D_s^X$ as simply $D_s$ when the context is clear. We shall denote the {\em domain} of an element $x\in X$ by
$$
D^x = \{s\in S| sx\text{ exists}\}.
$$
Clearly $x\in D_s$ if and only if $s\in D^x$. Notice also that it follows from the definition that $x\in D_s$ if and only if $x \in D_{s's}$ for some $s'\in W(s)$.

\medskip

If $S$ is a group then an $E-$dense act $X$ is simply an $S-$set, while if $S$ is an inverse semigroup then an $E-$dense act is an inverse semigroup act defined by the Wagner-Preston representation $\rho:S\to {\cal I}(X)$ where $sx = \rho(s)(x)$ and $D_s = \dom(\rho(s))$ (see Example~\ref{preston-wagner-representation} below for a generalisation).

\smallskip

Let $X$ be an $E-$dense $S-$act and let $x \in X$.  We define the {\em stabilizer} of an element $x$ as the set $S_x=\{s\in S|sx=x\}$. The following is easy to establish.
\begin{lemma}\label{basic-lemma}
Let $S$ be an $E-$dense semigroup and $X$ an $E-$dense $S-$act. Let $s,t\in S, x,y\in X$. Then
\begin{enumerate}
\item $E\cap D^x\subseteq S_x$,
\item if $s'\in W(s)$ then $x\in D_{s'}$ if and only if $x\in D_{ss'}$,
\item if $s\in D^x$ then $sx=y$ if and only if there exists $s'\in W(s)\cap D^y$ such that $x = s'y$,
\item if $s,t\in D^x$ then $sx=tx$ if and only if there exists $s'\in W(s)$ such that $s't\in S_x$. In addition, any such $s'$ necessarily satisfies $s'\in D^{sx}$,
\end{enumerate}
\end{lemma}

\begin{example} {\rm (Wagner-Preston action)} \label{preston-wagner-representation}
Let $S$ be an $E-$dense semigroup with semilattice of idempotents $E$ and $X$ a set on which $S$ acts (on the left) via the representation $\rho:S\to {\cal T}(X)$. In other words the action on $X$ is a total action. For each $s\in S$ define
$$
D_s = \{x\in X|\text{ there exists }s'\in W(s), x=s'sx\} = \{s'sx|x\in X, s'\in W(s)\}
$$
and define an $E-$dense action of $S$ on $X$ by $s\ast x = sx$ for all $x\in D_s$.
\end{example}

To see that $\ast$ really is an $E-$dense action suppose that $x\in D_{st}$ so that there exists $(st)'\in W(st)$ such that $x=(st)'(st)x$. By Lemma~\ref{weak-inverse-product-lemma} there exists $s'\in W(s), t'\in W(t)$ such that $(st)' = t's'$ and so $x=t's'stx$. Then $t'tx = t'tt's'stx = t's'stx = x$ and so $x\in D_t$. In addition, $s'stx = s'stt's'stx = tt's'stx = tx$ and $tx\in D_s$. Conversely, suppose that $x\in D_t$ and $tx\in D_s$. Then there exists $t'\in W(t), s'\in W(s)$ such that $x=t'tx$ and $tx = s'stx$ and so $x=t's'stx \in D_{st}$. Clearly, $(st)\ast x = s\ast(t\ast x)$. Finally, if $x,y\in D_s$ and $s\ast x=s\ast y$ then
there exists $s'\in W(s), s^\ast\in W(s)$ such that $x = s'sx, y=s^\ast sy$ and such that $sx=sy$. Hence
$$
x=s'sx = s'sy = s'ss^\ast sy = s^\ast ss'sy = s^\ast ss'sx = s^\ast sx=s^\ast sy=y.
$$
In addition, if $x\in D_s$ then
$x = s'sx$ for some $s'\in W(s)$, and so letting $(s's)' = s's\in W(s's)$ then
$$
(s's)'(s's)x = s's x = x
$$
and $x \in D_{s's}$ as required. Hence $\ast$ satisfies the conditions of an $E-$dense action.

\smallskip
In particular, we can take $X=S$, or indeed any left ideal of $S$, with (total) action given by the multiplication in $S$.

\bigskip

A element $x$ of $X$ is said to be {\em effective} if $D^x\ne\emptyset$. An $E-$dense $S-$act $X$ is {\em effective} if all its elements are effective. An $E-$dense $S-$act is {\em transitive} if for all $x,y \in X$, there exists $s \in S$ with $y = sx$. Notice that this is equivalent to $X$ being {\em locally cyclic} in the sense that for all $x,y\in X$ there exists $z\in X, s,t\in D^z$ with $x=sz, y=tz$. We shall consider transitive acts in more detail in Section~\ref{transitive-section}.

\smallskip

If $X$ is an $E-$dense $S-$act and $Y$ is a subset of $X$ then we shall say that $Y$ is an $E-$dense {\em $S-$subact} of $X$ if for all $s\in S, y\in D_s^X\cap Y \Rightarrow sy\in Y$. Notice that this makes $Y$ an $E-$dense $S-$act with the action that induced from $X$ and $D_s^Y = D_s^X\cap Y$ for all $s\in S$.

\smallskip

Let $X$ and $Y$ be two $E-$dense $S-$acts. A function $f : X\to Y$ is called an {\em ($E-$dense) $S-$map} if for all $s \in S$, $x\in D_s^X$ {\em if and only if} $f(x)\in D_s^Y$ and then $f(sx) = sf(x)$.

\smallskip

For example, if $Y$ is an $S-$subact of an $E-$dense $S-$act $X$, then the inclusion map $\iota : Y \to X$ is an $S-$map.

\smallskip

Let $x \in X$ and define the {\em $S-$orbit} of $x$ as
$$
Sx = \{sx | s \in D^x\}\cup\{x\}.
$$
Notice that if $x$ is effective, then there exists $s\in D^x$ and so for any $s'\in W(s)\cap D^{sx}, x = s'sx\in \{sx|s\in D^x\}=Sx$. However if $x$ is not effective then $\{sx|s\in D^x\}=\emptyset$ and $Sx=\{x\}$. Notice also that $Sx$ is an $E-$dense $S-$subact of $X$ (the {\em subact generated by $x$}) and that the action is such that, for all $tx \in Sx$ and all $s \in S$, $tx \in D_s^{Sx}$ if and only if $x \in D_{st}^X$ and in which case $s(tx) = (st)x$. Then we have

\begin{lemma}
For all $x \in X$, if $x$ is effective then so is $Sx$, in which case $Sx$ is a transitive $E-$dense $S-$act. 
Conversely, if an $E-$dense $S-$act is effective and transitive then it has only one $S-$orbit.
\end{lemma}

\begin{proof}
Suppose that $x$ is effective. Then let $s\in D^x$ so that $sx \in Sx$, and notice that there exists $s'\in W(s)\cap D^{sx}$. Therefore $ss'(sx) = sx \in Sx$ and hence $Sx$ is effective. If $y = s_1x$ and $z = s_2x$ then put $t = s_1s_2'$, where $s_2'\in W(s_2)\cap D^z$, to get $y = tz$ (if $x=y\ne z$ then take $t=s_2'$; if $x=z\ne y$ then take $t=s_1$ while if $x=y=z$ take $t=s's$ where $s\in D^x, s'\in W(s)\cap D^{sx}$).

The converse is easy. Note that in this case $Sx = \{sx | s \in D^x\}$.
\end{proof}

Notice that $Sx=Sy$ if and only if $y\in Sx$ and so the orbits partition $X$.

\smallskip

Recall that Green's ${\cal L}-$relation is given by $a{\cal L}b$ if and only if $S^1a=S^1b$. As is normal, we shall denote the ${\cal L}-$class containing $a$ as $L_a$.

\begin{proposition}\label{idempotent-orbit-proposition}
Let $S$ be an $E-$dense semigroup with semilattice of idempotents $E$ and consider $S$ as an $E-$dense $S-$act with the Wagner-Preston action.
\begin{enumerate}
\item If $e\in E$ then $S_e = e\omegah$ and $Se = L_e$.
\item For all $s\in S$ and for all $s'\in W(s)$, $S_s\subseteq (ss')\omega$ and $Ss\subseteq L_s$. In addition $S_s=(ss')\omega$ for some $s'\in W(s)$ if and only if $s$ is regular, in which case $Ss = L_s$ and we can assume that $s'\in V(s)$.
\item For all $se \in Se\ ( \text{the orbit of }e), S_{se}=(ses')\omegah$ for some $s'\in W(s)$ and $Sse = L_{se}$.
\item For all $s\in S, s'\in W(s)$ it follows that $S_{s'} = (s's)\omegah$ and $Ss'=L_{s'}$.
\end{enumerate}
\end{proposition}
\begin{proof}
\begin{enumerate}
\item If $t \in S_e$ then there exists $t'\in W(t)$ such that $e=t'te=t'e$. Hence $e = e(t't) = (tet')t$ and so $e\leh t$ and $t\in e\omegah$.

\smallskip

Conversely, if $e \leh t$ then there exists $f,g\in E$ such that $e = ft=tg$ and it is easy to check that $e = et = te$. Since $e\in W(e)$ then there exist $e'\in W(e), t'\in W(t)$ such that $e = e't'$ and so $t'te = t'tee't' = ee't' = e$ and $t\in D^e$ and since $te=e$ then $t\in S_e$ as required.

If $te \in Se$ then there exists $t'\in W(t)$ such that $e = t'te$. Hence $te{\cal L}e$. On the other hand, if $s{\cal L}e$ then there exist $u,v\in S^1$ such that $us=e, ve=s$, from which we deduce that $se=s$. If $s=e$ then obviously $s\in Se$, otherwise note that $s'=eu\in W(s)$ and so since $s'se = euse = e$ then $s\in D^e$ and $s=se\in Se$ and hence the orbit of $e$ and ${\cal L}-$class containing $e$ coincide.

\item Let $t\in S_s$ so that there exists $t'\in W(t)$ such that $s=t'ts$ and $ts=s$. Then
$$
ss' = tss' = t(t'tss') = (tss't')t
$$
and so $t\in (ss')\omega$. If $rs\in Ss$ then there exists $r'\in W(r)$ such that $s=r'rs$ and so $rs{\cal L}s$ and $Ss\subseteq L_s$.

If $(ss')\omega\subseteq S_s$ then in particular $ss'\in S_s$ and so $s'\in D^s$. Hence there exists $s'^\ast\in W(s')$ such that $s=s'^\ast s's$ and so $s=s'^\ast\in W(s')$ which means that $s$ is regular and $s'\in V(s)$. Conversely, if $s$ is regular then there exists $s'\in V(s)$ and so $ss'\in S_s$. Hence $(ss')\omega\subseteq S_s$ since $S_s$ is closed. In this case, since ${\cal L}$ is a right congruence, then for any $s'\in V(s)$
\begin{gather*}
ts\in Ss\iff t\in D^{s}\iff t\in D^{ss'}\iff tss'\in Sss'\iff tss'{\cal L}ss'\iff ts{\cal L}s
\end{gather*}
Hence $Ss = L_s$.
\item This follows from part (2) since $se$ is regular.

\item Since $s'$ is regular then from part(2) there exists $s'^\ast\in V(s')$ such that $S_{s'}=(s's'^\ast)\omega$ and $Ss'=L_{s'}$. But from Lemma~\ref{weak-inverses-lemma}, $s'^\ast = ss's$ and the result follows.
\end{enumerate}
\end{proof}

Let $x\in X$ and set $E^x = E\cap D^x$. In analogy with group theory, and following~\cite{funk-10}, we shall say that an $E-$dense $S-$act $X$ is {\em locally free} if for all $x\in X, S_x= (E^x)\omegah$.

\begin{theorem}
Let $S$ be an $E-$dense semigroup with semilattice of idempotents $E$ and let $X$ be an $E-$dense $S-$act. Then $X$ is locally free if and only if for all $x\in X, s,t\in D^x$, whenever $sx=tx$ there exists $e\in S_x$ such that $se=te$.
\end{theorem}
\begin{proof}
Suppose that $S$ acts locally freely on $X$ and that $sx=tx$ for some $x\in X, s,t\in D^x$. Then there exists $s'\in W(s)$ with $s't\in S_x$ and so there exist $f,g\in E, e\in E^x$ with $(s't)g = f(s't) = e$. Since $e\in W(e)$ then there exist $t'\in W(t), s'^\ast\in W(s'), f'\in W(f)$ with $e=t's'^\ast f'$. Hence since $ge=eg=e$, $t'te = e$ and $ss's'^\ast = s'^\ast$ then
$$
se = sfs'tt's'^\ast f' = tt'sfs'tt's'^\ast f' = tt'ss'tgt's'^\ast f' = tt'tgt's'^\ast f' = tt'tge = te.
$$
Conversely, suppose that $s\in S_x$ so that $sx=x$. Then there exists $s'\in W(s)$ such that $sx=s'sx$. By assumption there exists $e\in E^x$ such that $se=s'se = es's$. But $ses' = es'ss' = es'$ and so $se=(ses')s\in E^x$ and hence $s \in (E^x)\omegah$ and $X$ is locally free.
\end{proof}

\subsection{Graded actions}
Let $S$ be an $E-$dense semigroup with semilattice of idempotents $E$. We can consider $E$ as an $E-$dense $S-$act with action given by the {\em Munn representation} on $E$. In more detail, let $e\in E$ and let $[e]$ denote the {\em order ideal} generated by $e$. This is the set
$$
[e]=\{s\in S|s\leh e\}=\{s\in E|s=es=se\}=eE.
$$
The second equality is easy to establish on observing that $[e]\subseteq E$ (see~\cite[Lemma 2.1]{mitsch-94}).

\smallskip

\begin{lemma}
Let $S$ be an ($E-$dense) semigroup with semilattice of idempotents $E$. Then for all $e\in E$, $[e] = W(e)$.
\end{lemma}
\begin{proof}
If $s\in [e]$ then $s=es=se$ and so $ses = s^2 = s$. Hence $[e]\subseteq W(e)$. Conversely, if $s\in W(e)$ then $ses=s$ and so $es,se\in E$. Hence $s = ses = sees \in E$. Consequently, $s=se=es \in [e]$ and so $[e] = W(e)$.
\end{proof}

\smallskip

The action of $S$ on $E$ is given as follows. For each $s\in S$ define $D_s=\bigcup_{s'\in W(s)}[s's]$ and for each $x \in D_s\subseteq E$ define an action $s\ast x = sxs'$ where $x \in [s's]$ with $s'\in W(s)$. Notice that if $x\in [s^\ast s]\cap [s's]$ and $s',s^\ast\in W(s)$ then $x=s'sx =xs's$ and $x=s^\ast sx = xs^\ast s$. Consequently,
$$
sxs' = sxs^\ast ss' = ss'sxs^\ast=sxs^\ast.
$$
So the action is well-defined. Notice then that if $x\in D_{st}$ then $x\leh (st)'(st)$ for some $(st)'\in W(st)$. Since, by Lemma~\ref{weak-inverse-product-lemma}, $W(st)=W(t)W(s)$ then there exists $s'\in W(s), t'\in W(t)$ such that $x=t's'stx=xt's'st$. Hence $xt't = t'tx = t'tt's'stx = t's'sx = x$ and so $x\in D_t$.

In addition $s's(txt') = (txt')s's = txt's'stt' = txt'$ and so $t\ast x\in D_s$.

\medskip

Conversely, suppose that $x\in D_t$ and $t\ast x\in D_s$ so that $x = t'tx = xt't$ for some $t'\in W(t)$ and that $txt'=s'stxt'=txt's's$ for some $s'\in W(s)$. Then
$$
xt's'st = t's'stx = t's'stxt't=t'txt't=x
$$
and so $x\in D_{st}$.

\medskip

Now, if $x\in D_{st}$ then for some $(st)'\in W(st), s'\in W(s), t'\in W(t)$ we have
$$
(st)\ast x = (st)x(st)' = stxt's'=s\ast(txt') = s\ast(t\ast x).
$$
If $s\ast x = s\ast y$ then $x = s'sx=xs's, y = s^\ast sy=ys^\ast s$ and $sxs'=sys^\ast$  for some $s',s^\ast\in W(s)$. Hence $x = s'sxs's=s'sys^\ast s$ and so $x\leh y$. Dually $y\leh x$ and so $x=y$.

Finally, if $x\in D_s$ then there exists $s'\in W(s)$ such that $x=s'sx =xs's$. Since $s's\in W(s's)$ then it easily follows that $x \in D_{s's}$. Consequently we have established that $E$ is an $E-$dense $S-$act with action given as above.

\medskip

Let $X$ be an $E-$dense $S-$act. Following~\cite{steinberg} we say that the action is {\em graded} if there exists a function $p : X\to E$ such that for all $e\in E, D_e = p^{-1}([e])$, and refer to $p$ as the {\em grading}.

\medskip

\begin{lemma}\label{graded-minimum-idempotent-lemma}
Let $S$ be an $E-$dense semigroup with semilattice of idempotents $E$, and $X$ a graded $E-$dense $S-$act. Then $X$ is effective and for all $x \in X, p(x)$ is the minimum idempotent in $S_x$.
\end{lemma}

\begin{proof}
Suppose that $X$ is graded with grading $p:X\to E$ and let $x\in X$. Then as $x \in p^{-1}([p(x)]) = D_{p(x)}$ for all $x\in X$ it follows that $X$ is effective. Notice also that $p(x)\in S_x\cap E$.  Suppose that there exists $e\in S_x\cap E$. 
Then $x\in D_e=p^{-1}([e])$ and so $p(x) \in [e]$. Hence $p(x)\leh e$ as required.
\end{proof}

The following is fairly clear.

\begin{proposition}\label{free-graded-proposition}
Let $S$ be an $E-$dense semigroup with semilattice of idempotents $E$, and $X$ a graded $E-$dense $S-$act with grading $p:X\to E$. Then $X$ is locally free if and only if for all $x\in X, S_x = p(x)\omegah$.

\medskip

Conversely, if $X$ is an $E-$dense $S-$act with the property that for all $x\in X$ there exists $e_x\in E$ with $S_x=e_x\omegah$, then $X$ is locally free and graded with grading $p:X\to E$ given by $p(x) = e_x$.
\end{proposition}
\begin{proof}
Suppose that $X$ is locally free. If $s \in (E^x)\omegah$ then there exists $e\in E^x$ such that $e\leh s$. Then since $e\in S_x$ it follows that $p(x)\leh e\leh s$. On the other hand, it is clear that $p(x)\omegah\subseteq (E^x)\omegah\subseteq S_x$ and so $S_x = p(x)\omegah$.

\smallskip

If $S_x = p(x)\omegah$ then clearly $S_x\subseteq (E^x)\omegah$. But $(E^x)\omegah\subseteq S_x$ and so $X$ is locally free.

\smallskip

The converse follows easily from Lemma~\ref{graded-minimum-idempotent-lemma}.
\end{proof}

It follows from Lemma~\ref{graded-minimum-idempotent-lemma} that the grading function $p$ is unique. Notice also that if $p(x)'\in W(p(x)) \cap D^{p(x)x}$ then $p(x)'p(x)\in S_x$ and so $p(x)'\in S_x$. Consequently $p(x)p(x)'\in S_x$. Moreover $p(x)\leh p(x)'p(x),p(x)p(x)'$ from which we easily deduce that $p(x) = p(x)'p(x)=p(x)p(x)'$. But then $p(x)' = p(x)'p(x)p(x)' = p(x)p(x)' = p(x)$.

\begin{lemma}\label{graded-property-2-lemma}
Let $S$ be an $E-$dense semigroup with semilattice of idempotents $E$ and $X$ a graded $E-$dense $S-$act with grading $p$. Then for $x\in D_s$, if $s's = p(x)$ for $s'\in W(s)$ then $ss'=p(sx)$.
\end{lemma}
\begin{proof}
Suppose that $s's = p(x)$. Then $x\in D_{s's}$ and so $x\in D_s$. In addition, $sx\in D_{s'}$ and so $ss' \in S_{sx}$ which means that $p(sx)\leh ss'$. Now $s's = s'ss's\geh s'p(sx)s$ (since $E$ is a semilattice). But since $s'p(sx)s\in S_x\cap E$ then by Lemma~\ref{graded-minimum-idempotent-lemma}, $p(x) = s's = s'p(sx)s$ and so $ss' = sp(x)s' =  ss'p(sx)ss'$, or in other words $ss'\leh p(sx)$. But as $ss'\geh p(sx)$ then $p(sx) = ss'$ as required.
\end{proof}

\begin{corollary}\label{graded-munn-action-corollary}
Let $S$ be an $E-$dense semigroup with semilattice of idempotents $E$ and $X$ a graded $E-$dense $S-$act with grading $p$. Let $s \in S$ and $x\in D_s$. Then for all $s'\in W(s)\cap D^{sx}$, $p(sx) = sp(x)s'$.
\end{corollary}

\begin{proof}
Let $t=sp(x)$ and let $s'\in W(s)\cap D^{sx}$. Then $t'=p(x)s'\in W(sp(x))=W(t)$. Hence $t't = p(x)s'sp(x) = s'sp(x) = p(x)$ as $s's\in S_x$. In addition, $tx = sp(x)x = sx$ and so by Lemma~\ref{graded-property-2-lemma}, $p(sx) = p(tx) = tt' = sp(x)p(x)s' = sp(x)s'$ as required.
\end{proof}

\medskip

\begin{proposition}[Cf. {\cite[Proposition 1.1]{steinberg}}]
Let $S$ be an $E-$dense semigroup with semilattice of idempotents $E$ and $X$ a graded $E-$dense $S-$act with grading $p$ and let $s\in S$. Then $D_s = \bigcup_{s'\in W(s)}p^{-1}([s's])$ and $sX = \{sx|x\in D_s\} = \bigcup_{s'\in W(s)}p^{-1}([ss'])$.
\end{proposition}
\begin{proof}
Let $s\in S, x\in D_s, s'\in W(s)\cap D^{sx}$. Then $x\in D_{s's}= p^{-1}([s's])$ and so $D_s = \bigcup_{s'\in W(s)}p^{-1}([s's])$. Since $p(sx) = sp(x)s' = (sp(x)s')(ss') = (ss')(sp(x)s')\leh ss'$ then $p(sx)\in [ss']$ and so $sx\in p^{-1}([ss'])$. Conversely, if $y\in p^{-1}([ss']) = D_{ss'}$ then $y = ss'y = sx$ where $x = s'y$. Hence $sX = \bigcup_{s'\in W(s)}p^{-1}([ss'])$.
\end{proof}

\medskip

\begin{theorem}
Let $S$ be an $E-$dense semigroup with semilattice of idempotents $E$ and $X$ an $E-$dense $S-$act. The following are equivalent.
\begin{enumerate}
\item $X$ is a graded $E-$dense $S-$act,
\item there exists an $E-$dense $S-$map $f:X\to E$,
\item $X$ is an effective $E-$dense $S-$act and for all $x \in X$, $S_x$ contains a minimum idempotent.
\end{enumerate}
\end{theorem}

\begin{proof}
(1) $\implies(2)$. If $x\in D_s^X$ then from Corollary~\ref{graded-munn-action-corollary}, for all $s'\in W(s)\cap D^{sx}$, $p(sx) = sp(x)s' = s\ast p(x)$ and $p(x) \in \bigcup_{s'\in W(s)}[s's]$. Hence $p(x)\in D_s^E$. Conversely, if $p(x)\in D_s^E$ then $p(x) \in \bigcup_{s'\in W(s)}[s's]$ and so there exists $s'\in W(s)$ such that $x\in p^{-1}([s's]) \subseteq D_s^X$. In addition $s\ast p(x)=sp(x)s' = p(sx)$ and it follows that $p$ is an $S-$map.

(2) $\implies$ (3). Suppose that $X$ is an $S-$act with an $E-$dense $S-$map $f:X\to E$ and let $x\in X$. Then as $f(x) \in  D_{f(x)}^E$, it follows that $x\in D_{f(x)}^X$ and $X$ is effective. Notice also that $f(x)\in S_x\cap E$.  Suppose then that there exists $e\in S_x\cap E$. Then $x\in D_e^X$ so $f(x)\in D_e^E=[e]$ and so $f(x)\leh e$ as required.

(3) $\implies$ (1).  If $X$ is an effective $E-$dense $S-$act and for all $x\in S, S_x$ contains a minimum idempotent, say $e_x$, then define a function $p:X\to E$ by $p(x) = e_x$. Suppose then that $e\in E$ and $x\in D_e^X$. Then $e\in S_x\cap E$ and so $p(x)\leh e$. Hence $p(x)\in [e]$ or in other words $x\in p^{-1}([e])$ and so $D_e^X\subseteq p^{-1}([e])$. On the other hand, if $x\in p^{-1}([e])$ then $p(x)\in[e]$ and so $p(x) = p(x)e$ and since $x\in D_{p(x)}$ then $x\in D_e$ as well. Hence $D_e=p^{-1}([e])$ and $p$ is a grading.
\end{proof}
It is easy to check that $E$ is a graded $E-$dense $S-$act with grading $1_E:E\to E$, the identity function. The following is clear.
\begin{corollary}
If $X$ is an $E-$dense $S-$act and $E$ is finite then $X$ is graded.
\end{corollary}

\medskip

Let $(X,p)$ and $(Y,q)$ be graded $E-$dense $S-$acts with grading functions $p$ and $q$. A {\em graded morphism} is an $E-$dense $S-$map $f : X \to Y$ such that $qf = p$. It is clear that graded $E-$dense $S-$acts and graded morphisms form a category and that $(E,1_E)$ is a terminal object in this category.

\subsection{Transitive $S-$acts}\label{transitive-section}
An $E-$dense $S-$act is called {\em indecomposable} if it cannot be written as the coproduct (i.e. disjoint union) of two other $E-$dense $S-$acts. In particular, a transitive $S-$act is easily seen to be indecomposable. Conversely, if $X$ is indecomposable, then suppose that $Y=X\setminus Sx\ne\emptyset$ for some $x\in X$. Then $Y$ cannot be a subact of $X$ as $X$ is indecomposable, so there exists $y\in Y,s \in S$ with $sy\in Sx$ and hence $y\in Sx$, a contradiction. Therefore $X=Sx$ is transitive. The transitive $S-$acts are therefore the `building blocks' of $E-$dense $S-$acts. In this section, we restrict our attention, in the main, to those $E-$dense semigroups where $E$ is a semilattice.

\medskip

Suppose that $S$ is an $E-$dense semigroup and that $H$ is a subsemigroup of $S$. If for all $h\in H, W(h)\cap H\ne\emptyset$ then we will refer to $H$ as an {\em $E-$dense subsemigroup of $S$}. For example, if $E$ is a band then $E$ is an $E-$dense subsemigroup of $S$.
\begin{lemma}\label{e-dense-closed-lemma}
Let $S$ be an $E-$dense semigroup with semilattice of idempotents $E$ and let $H$ be an $E-$dense subsemigroup of $S$. Then $H\omegah$ is an $E-$dense subsemigroup of $S$.
\end{lemma}
\begin{proof}
Suppose that $x,y \in H\omegah$ so that there exist $a,b\in H$ such that $a\leh x, b\leh y$. In addition, there exists $a'\in W(a)\cap H, b'\in W(b)\cap H$. Hence there exists $e,f,g,h\in E$ such that $a = xe=fx, b=yg=hy$. Let $x'\in W(x),f'\in W(f), y'\in W(y),h'\in W(h)$ be such that $a'=x'f'\in W(a), b'=y'h'\in W(b)$. Then
\begin{gather*}
(xy)(y'h'hx'f'fxy) = (xyy'h'hx'f'f)(xy) = (f'f)(xhyy'h'x')(xhy) = \\
(f'f)(fxhyy'h'x')(xhy) = (fxhyy'h'x')(f'fxhy) = abb'a'ab\in H
\end{gather*}
and so $xy\in H\omegah$ and $H\omegah$ is a subsemigroup of $S$. Now suppose that $x \in H\omegah$ so that there exists $h\in H$ and $e,f\in E$ such that $h = ex = xf$. Suppose also that $h'\in W(h)\cap H$ so that there exists $x'\in W(x), f'\in W(f)$ such that $h' = f'x'$. Then
$$
x'(xff'x') = (x'xff')x' = ff'x'xff'x' = f'x'xfff'x' = h'hh'=h'\in H
$$
and so $x'\in H\omegah$ and $H\omegah$ is an $E-$dense subsemigroup of $S$.
\end{proof}

\begin{lemma}\label{idempotent-closed-lemma}
Let $S$ be an $E-$dense semigroup with semilattice of idempotents $E$ and let $H$ be an $E-$dense subsemigroup of $S$. Let $x,y\in S, x'\in W(x), y'\in W(y), e\in E$. Then
\begin{enumerate}
\item if $x'ex\in H\omegah$ then $x'x\in H\omegah$;
\item if $x'ey, y'y\in H\omegah$ then $x'y\in H\omegah$.
\end{enumerate}
\end{lemma}
\begin{proof}
Let $x,y\in S, x'\in W(x), y'\in W(y), e\in E$. Notice that by Lemma~\ref{e-dense-closed-lemma}, $H\omegah$ is an $E-$dense subsemigroup of $S$.
\begin{enumerate}
\item By assumption there exists $f,g\in E,a\in H$ such that $a = (x'ex)f = g(x'ex)$. Consequently
$$
a = x'exf = (x'x)(x'exf) = (x'exf)(x'x)
$$
and so $a\leh x'x$ and $x'x\in H\omegah$.
\item By assumption there exists $f,g\in E,a\in H$ such that $a = (x'ey)f = g(x'ey)$. Consequently
$$
ay'y = x'eyfy'y = x'eyy'yf = (x'y)(y'eyf) = x'xx'eyfy'y = (x'eyfy'x)(x'y)
$$
and so $ay'y\leh x'y$ and $x'y\in H\omegah$ as required.
\end{enumerate}
\end{proof}

\begin{proposition}
Let $S$ be an $E-$dense semigroup with semilattice of idempotents $E$ and let $H$ be an $E-$dense subsemigroup of $S$. Then the following are equivalent
\begin{enumerate}
\item $H$ is $\omegah-$closed in $S$;
\item $H$ is unitary in $S$;
\item $H$ is $\omegam-$closed in $S$.
\end{enumerate}
\end{proposition}
\begin{proof}
$(1)\implies(2)$. Suppose that $H$ is $\omegah-$closed in $S$ and suppose that $hs = h_1$ for some $s\in S, h,h_1\in H$. Then there exists $h'\in W(h)\cap H, h_1'\in W(h_1)\cap H$ and so there exist $s'\in W(s), h^\ast\in W(h)$ such that $h_1'=s'h^\ast\in W(hs)$. Then
$$
s(s'h'hh^\ast hs) = (ss'h'hh^\ast h)s = h'hss'h^\ast h_1 = h'h_1h_1'h_1\in H
$$
and so $s \in H\omegah = H$. Consequently $H$ is left unitary in $S$. The right unitary property follows in a similar way.

$(2)\implies(3)$. Suppose $H$ is unitary in $S$ and that $s\gem h$ for $h\in H$. Then there exist $x,y\in S$ with $h=xs=sy, xh=hy=h$. Let $h'\in W(h)\cap H$ and notice that $h'hyh'h = h'hh'h = h'h\in H$. Therefore $y\in H$ and so $s\in H$ and $H$ is $\omegam-$closed in $S$.

$(3)\implies(1)$. As $H\subseteq H\omegah \subseteq H\omegam$ then this is clear.
\end{proof}

In view of the above result, we shall simply say that a set $A$ is {\em closed} if it is $\omegah-$closed.

\medskip

We briefly review Schein's theory of partial congruences when applied to $E-$dense semigroups which have a semilattice of idempotents (see~\cite[Chapter 7]{clifford} or \cite[Chapter 5]{howie-95} for more details of the case for inverse semigroups).

\smallskip

Let $T \subseteq S$ be sets and suppose that $\rho$ is an equivalence on $T$. Then we say that $\rho$ is a {\em partial equivalence} on $S$ with domain $T$. It is easy to establish that $\rho$ is a partial equivalence on $S$ if and only if it is symmetric and transitive. If now $T$ is an $E-$dense subsemigroup of an $E-$dense semigroup $S$ and if $\rho$ is left compatible with the multiplication on $S$ (in the sense that for all $s\in S, (u,v)\in\rho$ either $su, sv \in T$ or $su, sv\in S\setminus T$ and $(su,sv)\in\rho$ in the former case) then $\rho$ is called a {\em left  congruence} on $S$ and the set $T/\rho$ of $\rho-$classes will often be denoted by $S/\rho$.

\begin{theorem}\label{piK} Let $H$ be a closed $E-$dense subsemigroup of an $E-$dense semigroup $S$ and suppose that $E$ is a semilattice. Define
$$
\pi_H = \{(s,t)\in S\times S| \exists s'\in W(s), s' t\in H\}. 
$$

Then $\pi_H$ is a left partial congruence on $S$ and the domain of $\pi_H$ is the set $D_H = \{s \in S | \exists s'\in W(s), s's \in H\}$.

\smallskip

The (partial) equivalence classes are the sets $(sH)\omegah$ for $s \in D_H$. The set $(sH)\omegah$ is the equivalence class that contains $s$ and in particular $H$ is one of the $\pi_H-$classes.
\end{theorem}
\begin{proof}
It is clear that $\pi_H$ is reflexive on $D_H$. Notice first that if there exists $s'\in W(s)$ such that $s't\in H$ then $s'ss't\in H$ and so since $H$ is unitary, $s's\in H$. Suppose then that $(s,t)\in \pi_H$. Then there exists $s'\in W(s), t'\in W(t)$ such that $s's,t't, s't\in H$. Let $t^\ast\in W(t), s'^\ast\in W(s')$ be such that $t^\ast s'^\ast \in W(s't)\cap H$. Then let $x = t's$ and $x'=(s't)(t't)\in W(x)$ so that $x'^\ast = (t't)(t^\ast s'^\ast)\in W(x')\cap H$. By Lemma~\ref{weak-inverses-lemma} $x'^\ast=(t't)x'^\ast \leh x$ and so $t's=x\in H$ and hence $\pi_H$ is symmetric.

Now suppose that $(s,t), (t,r)\in \pi_H$. Then there exists $s'\in W(s), t'\in W(t), r'\in W(r)$ such that $s's, s't, t't, t'r, r'r \in H$. Consequently
$$
(s'r)(r'tt'ss'r) = (s'rr'tt's)(s'r) = s'tt'rr'ss'r = s'tt'rr'r \in H
$$
and so $s'r\in H\omegah=H$ and $\pi_H$ is transitive.

\smallskip

Suppose that $(s,t)\in \pi_H$ and that $r\in S$ and suppose further that $rs,rt\in D_H$. Then there exists $s'\in W(s), t'\in W(t)$ such that $s's, t't, s't\in H$. Further, there exists $s^\ast\in W(s), t^\ast\in W(t), r',r^\ast\in W(r)$ such that $s^\ast r'rs, t^\ast r^\ast rt\in H$. From Lemma~\ref{weak-inverses-lemma}, $s^\ast ss'\in W(s)$ and so $(rs)'(rt) = s^\ast ss'r'rt = s^\ast r'rss't \in H$. Hence $(rs,rt)\in \pi_H$ and $\pi_H$ is a left partial congruence on $S$. 

\smallskip

Now suppose that $s\in (tH)\omegah$ where $t't\in H$. Then there exists $h\in H$ such that $th\leh s$ and so there exists $e,f\in E$ such that $th = se=fs$. Hence $t'th = t'se=t'fs = t'tt'fs=(t'ft)t's$ and so $t'th\leh t's$ and $t's\in H\omegah = H$. Consequently $s\in[t]_{\pi_H}$. On the other hand, if $s\in[t]_{\pi_H}$ then there exists $s'\in W(s), t'\in W(t)$ such that $s's,t't,s't\in H$. Hence there exists $t^\ast\in W(t), s'^\ast \in W(s')$ such that $t^\ast s'^\ast \in W(s't)\cap H$. Now by Lemma~\ref{weak-inverses-lemma}, $tt^\ast s'^\ast\leh s$ and hence $s\in (tH)\omegah$.

\smallskip

Finally, if $s\in D_H$ then there exists $s'\in W(s)$ such that $s's\in H$ and so $s(s's) = (ss')s \in sH$ and hence $s\in (sH)\omegah$. In particular, for all $h_1,h_2\in H$ we see that $h_1\pi_H h_2$ and so $H=H\omegah=(hH)\omegah$ for all $h\in H$ is an equivalence class.
\end{proof}

The sets $(sH)\omegah$, for $s \in D_H$, are called the {\em left $\omegah-$cosets} of $H$ in $S$. The set of all left $\omegah-$cosets is denoted by $S/H$. Notice that $(sH)\omegah$ is a left $\omegah-$coset if and only if there exists $s'\in W(s)$ such that $s's\in H$. The following is then immediate.

\begin{proposition}\label{equivalence-class-proposition}
Let $H$ be a closed $E-$dense subsemigroup of an $E-$dense semigroup $S$ in which $E$ is a semilattice, and let $(aH)\omegah, (bH)\omegah$ be left $\omegah-$cosets of $H$. Then the following statements are equivalent:
\begin{enumerate}
\item $(aH)\omegah = (bH)\omegah$;
\item $a\pi_H b$ that is, there exists $b'\in W(b)$,  $b'a\in H$;
\item $a \in (bH)\omegah$;
\item $b \in (aH)\omegah$.
\end{enumerate}
\end{proposition}

\medskip

\begin{lemma}\label{piK_lemma}
With $H$ and $S$ as in Theorem~\ref{piK},
\begin{enumerate} 
\item precisely one left $\omegah-$coset, namely $H$, contains idempotents,
\item each left $\omegah-$coset is closed,
\item $\pi_H$ is left cancellative i.e. $xa\pi_H xb$ implies that $a\pi_H b$,
\item $((st)H)\omegah$ is an $\omegah-$coset if and only if $(tH)\omegah$ and $(s((tH)\omegah))\omegah$ are $\omegah-$cosets and then $(s((tH)\omegah))\omegah = ((st)H)\omegah$.
\end{enumerate}
\end{lemma}
\begin{proof}
\begin{enumerate}
\item If $e$ is an idempotent contained in an $\omegah-$coset then there exists $e'\in W(e)$ with $e'e\in H$. As $e'e\leh e$ then $e\in H\omegah=H$. As $H$ is an $E-$dense subsemigroup of $S$ then for each $h\in H$ there exists $h'\in W(h)\cap H$ and so $h'h\in H$ and hence $E(H)\ne\emptyset$.

\item This is clear.

\item Suppose that $(xa,xb)\in\pi_H$ so that there exists $(xa)'\in W(xa), (xb)'\in W(xb)$ such that $(xa)'(xa), (xa)'(xb)\in H$. Then there exists $x'\in W(x), a'\in W(a)$ such that $a'x'xb\in H$ and $a'x'xa\in H$. It follows from Lemma~\ref{idempotent-closed-lemma} that $a'a, a'b\in H$. Hence $a\pi_H b$.

\item Suppose that $((st)H)\omegah$ is an $\omegah-$coset, so that there exists $(st)'\in W(st)$ such that $(st)'st)\in H$. Then there exist $s'\in W(s), t'\in W(t)$ such that $t's' = (st)'$. Since $t's'st\in H$ and $H$ is closed then it follows from Lemma~\ref{idempotent-closed-lemma} that $t't\in H$ and so $(tH)\omegah$ is an $\omegah-$coset. If $x\in (s((tH)\omegah))\omegah$ then there exists $y\in (tH)\omegah$ such that $sy\leh x$ and so there exists $h\in H$ such that $th\leh y$. Hence there exist idempotents $e_1, e_2, f_1, f_2$ such that $sy = e_1x = xf_1$ and $th=e_2y = yf_2$. Now let $h'\in W(h)$ then $h't'\in W(th)=W(yf_2)$ and so there exists $f_2'\in W(f_2)$ and $y'\in W(y)$ such that $h't' = f_2'y'$. But $y's' \in W(sy) = W(xf_1)$ and so there exist $f_1'\in W(f_1)$ and $x'\in W(x)$ such that $y's' = f_1'x'$. Hence
\begin{gather*}
x(f_1f_2f_2'f_1'f_1f_2x'x) = (xf_1f_2f_2'f_1'f_1f_2x')x\\ = sthf_2'f_1'x'xf_1f_2 = sthh't's'sth \in (st)H
\end{gather*}
and so $x\in ((st)H)\omegah$.

On the other hand, suppose that $x\in ((st)H)\omegah$ so that there exists $e,f\in E, h\in H$ such that $ex=xf = sth$. Then $s(th)\leh x$ and $th\in tH\subseteq (tH)\omegah$ and hence $x\in (s((tH)\omegah))\omegah$.

Conversely, if $(tH)\omegah$ and $(s((tH)\omegah))\omegah$ are $\omegah-$cosets then as $t\in (tH)\omegah$ it follows that $st\in s((tH)\omegah)\subseteq (s((tH)\omegah))\omegah$. Which means that $(s((tH)\omegah))\omegah$ is the $\omegah-$coset containing $st$ and so equals $((st)H)\omegah$.
\end{enumerate}
\end{proof}

\smallskip

Notice that $S_x$ is a closed $E-$dense subsemigroup of $S$ for every $x\in X$. 

\smallskip

\begin{theorem}\label{S_x-closed-theorem}
For all $x \in X$, $S_x$ is either empty or a closed $E-$dense subsemigroup of $S$.
\end{theorem}
\begin{proof}
Assume that $S_x\ne\emptyset$. If $s,t \in S_x$ then $x=sx=s(tx)=(st)x$ and so $S_x$ is a subsemigroup. Also $sx = x$ implies that $x = s'x$ for any $s'\in W(s)\cap D^{sx}$ and so $S_x$ is an $E-$dense subsemigroup of $S$. Let $s \leh h$ with $s \in S_x$. Then there exist $e,f\in E$ such that $s = he = fh$. Consequently $e\in D^x$ and $h\in D^{ex}=D^x$ and so $hx = hex = sx = x$ which means that $h \in S_x$. Hence $S_x$ is $\omegah-$closed and so closed.
\end{proof}

\smallskip

From Lemma~\ref{piK_lemma} we can easily deduce the following important result.

\begin{theorem}\label{s-h-theorem}
If $H$ is a closed $E-$dense subsemigroup of an $E-$dense semigroup $S$ with semilattice of idempotents $E$ then $S/H$ is a transitive $E-$dense $S-$act with action given by $s\cdot X = (sX)\omegah$ whenever $X, sX \in S/H$. Moreover, it is easy to establish that $S_{H\omegah} = H$.
\end{theorem}

\begin{proof}
Let $X=(rH)\omegah$ be an $\omegah-$coset and suppose that $s,t\in S$ and that $X\in D_{st}$. Then by Lemma~\ref{piK_lemma}, $((st)X)\omegah = (st(rH)\omegah)\omegah$ is an $\omegah-$coset and
$$
(st)\cdot X = ((st)X)\omegah = ((str)H)\omegah = s\cdot(t\cdot X).
$$
In addition if $s\cdot X = s\cdot Y$ then $(sX)\omegah = (sY)\omegah$ and so Lemma~\ref{piK_lemma} $X = Y$. Now suppose that $X\in D_s$. We are required to show that there exists $s'\in W(s)\cap D^{sX}$. Suppose then that $X = (tH)\omegah$ so that $s\cdot X = ((st)H)\omegah$. Then there exists $(st)'\in W(st)$ such that $(st)'(st) \in H$. Hence there exist $s'\in W(s), t'\in W(t)$ such that $t's'st\in H$. Consequently, $t's'(ss's)s'st\in H$ and since $t's'(ss's)\in W(s'st)$ then there exists $(s'st)'\in W(s'st)$ such that $(s'st)'(s'st)\in H$ and so $((s'st)H)\omegah$ is an $\omegah-$coset of $H$ and $s'\in W(s)\cap D^{sX}$ as required.

If $(sH)\omegah$ and $(tH)\omegah$ are $\omegah-$cosets then there exist $s'\in W(s), t'\in W(t)$ such that $s's,t't\in H$. Now as $s'(ss's)t' = s'st'\in W(ts's)$ and as $s'st'ts's = t'ts's\in H$ then $((ts's)H)\omegah$ is an $\omegah-$coset. Moreover, as $(ts's)t't = tt'(ts's)\in tH$ then $ts's\in (tH)\omegah$ and so $(tH)\omegah = ((ts's)H)\omegah = (ts')\cdot((sH)\omegah)$ and $S/H$ is transitive.

Finally $S_{H\omegah} = \{s\in S| (sH)\omegah = H\omegah\}$. Hence $s\pi_H h$ for any $h\in H$ and so $s\in H$ as $H$ is an $\omegah-$coset.
\end{proof}

The converse of Theorem~\ref{s-h-theorem} is also true.
\begin{theorem}\label{orbit-stabilizer-theorem}
Let $S$ be an $E-$dense semigroup with semilattice of idempotents $E$, let $X$ be an effective, transitive $E-$dense $S-$act, let $x\in X$ and let $H=S_x$. Then $X$ is isomorphic to $S/H$. If $K$ is a closed $E-$dense subsemigroup of $S$ and if $X$ is isomorphic to $S/K$ then there exists $x\in X$ such that $K=S_x$.
\end{theorem}
\begin{proof}
Let $y\in X$ and notice that since $X$ is transitive then there exists $s\in S$ such that $y=sx$. Notice then that there exists $s'\in W(s)\cap D^{sx}$ such that $s's\in S_x=H$. Moreover if $y=tx$ for some $t\in S$ then $sx=tx$ and so $s't\in H$ and hence $(sH)\omegah = (tH)\omegah$. Therefore we have a well-defined map $\phi : X\to S/H$ given by
$$
\phi(y) = (sH)\omegah.
$$
Since $\phi(sx) = (sH)\omegah$ for all $(sH)\omegah\in S/H$ then $\phi$ is onto. If $(sH)\omegah = (tH)\omegah$ then $s't\in H=S_x$ and so $sx=tx$ and $\phi$ is a bijection. Finally, suppose $\phi(y) = (sH)\omegah$ and $t\in S$. Then $y\in D_t$ if and only if $t\in D^y$ if and only if $ts\in D^x$ if and only if there exists $(ts)'\in W(ts)$ such that $(ts)'(ts) \in S_x=H$ if and only if $((ts)H)\omegah$ is an $\omegah-$coset if and only if $\phi(y)=(sH)\omegah\in D_t$. In this case it is clear that $\phi(ty)=t\phi(y)$ and $\phi$ is an isomorphism.

\smallskip

By assumption there is an isomorphism $\theta:S/K \to X$. Let $x = \theta(K\omegah)$ so that $sx = \theta((sK)\omegah)$ for all $s\in D_K$. Notice also that since $\theta$ is an $S-$map then $s\in D^x$ if and only if $s\in D_K$. If $s\in S_x$ then $\theta(K\omegah) = \theta((sK)\omegah)$ and so $s\in K$ as $\theta$ is an isomorphism. On the other hand, if $s\in K$ then $sx = \theta((sK)\omegah) = \theta(K\omegah) = x$ and so $s\in S_x$ as required.
\end{proof}

Recall that $L_e$ denotes the ${\cal L}-$class containing $e$.

\begin{theorem}\label{free-transitive-graded-theorem}
Let $S$ be an $E-$dense semigroup with a semilattice of idempotents $E$ and let $X$ be a locally free, transitive, graded $E-$dense $S-$act with grading $p$. Then there exists $e\in E$ such that $X\cong Se\cong L_e$. Conversely, if $e\in E$ then the orbit $Se$ of $e$ in the $E-$dense $S-$act $S$ (with the Wagner-Preston action) is a locally free, transitive, graded $E-$dense $S-$act.
\end{theorem}
\begin{proof}
If $X$ is transitive then $X\cong Sx$ for some (any) $x\in X$. Using a combination of Theorem~\ref{s-h-theorem}, Proposition~\ref{free-graded-proposition} and Proposition~\ref{idempotent-orbit-proposition}, we deduce
$$
X\cong S/S_x=S/p(x)\omegah\cong S/S_{p(x)}=Sp(x)=L_{p(x)}.
$$
Conversely, the orbit $Se$ is clearly a transitive $E-$dense $S-$act. By Proposition~\ref{idempotent-orbit-proposition}, $S_{te} = (tet')\omegah$ and so by Proposition~\ref{free-graded-proposition} $Se$ is locally free and graded with grading $p:Se\to E$ given by $p(te) = tet'$.
\end{proof}

\smallskip

Lest $X$ be a graded $S-$act and let $x\in X$. If $p(x)\in D^S_s$ then there exists $s'\in W(s)$ such that $p(x)=s'sp(x)$ and so it follows that $x\in D_s^X$. Conversely, if $x\in D^X_s$ then there exists $s'\in W(s)\cap D^{sx}$ and so $p(x)\le s's$. Consequently $p(x)=s'sp(x)$ and $p(x)\in D^S_s$. Hence the map $Sp(x)\to Sx$ given by $sp(x)\mapsto sx$ is an $S-$map which is clearly onto. We have therefore demonstrated

\begin{proposition}
Let $S$ be an $E-$dense semigroup with a semilattice of idempotents $E$ and let $X$ be a graded $E-$dense $S-$act. Then $X$ is a quotient of a locally free graded $S-$act.
\end{proposition}

\smallskip

The question now arises as to when two transitive $E-$dense $S-$acts are isomorphic.

\begin{lemma}\label{conjugate-lemma}
Let $H$ be a closed $E-$dense subsemigroup of an $E-$dense semigroup $S$ with a semilattice of idempotents $E$. Let $(sH)\omegah$ be a left $\omegah-$coset of $H$ so that there exists $s'\in W(s)$ such that $s's \in H$. Then $sHs' \subseteq S_{(sH)\omegah}$.
\end{lemma}
\begin{proof}
Let $h\in H$ and consider $(shs')\cdot((sH))\omegah$. First notice that $(shs's)H\omegah$ is an $\omegah-$coset as for any $h'\in W(h)\cap H$, $s'sh's'\in W(shs's)$ and $(s'sh's')(shs's) \in H$. So $(shs')\cdot((sH))\omegah = ((shs's)H)\omegah = (sH)\omegah$ and so $sHs'\subseteq S_{(sH)\omegah}$.
\end{proof}

\medskip

If $H$ and $K$ are two closed $E-$dense subsemigroups of an $E-$dense semigroup $S$ with semilattice of idempotents $E$, then we say that $H$ and $K$ are {\em conjugate} if $S/H \cong S/K$ (as $E-$dense $S-$acts).

\begin{theorem}\label{clifford_conjugate}
Let $H$ and $K$ be closed $E-$dense subsemigroups of an $E-$dense semigroup $S$ with semilattice of idempotents $E$. Then $H$ and $K$ are conjugate if and only if there exist $s \in S, s'\in W(s)$ such that
$$
s'Hs \subseteq K {\rm\ and \ } sKs'\subseteq H.
$$

Moreover, any such element $s$ necessarily satisfies $ss'\in H, s's\in K$.
\end{theorem}
\begin{proof}
Suppose that $H$ and $K$ are conjugate. Then by Theorem~\ref{s-h-theorem} there is an $\omegah-$coset, $(sK)\omegah$ say, such that $S_{(sK)\omegah} = H$. So by Lemma~\ref{conjugate-lemma} there exists $s'\in W(s)$ such that $s's\in K$ and $sKs'\subseteq H$. Hence $ss'\in H$. In addition, for each $h\in H, (hsK)\omegah = (sK)\omegah$ and so $hs\pi_K s$. Consequently $s'hs \in K$ and so $s'Hs\subseteq K$ as required.

Conversely suppose there exist $s\in S, s'\in W(s)$ such that $s'Hs \subseteq K$ and $sKs'\subseteq H$. Then $ss'Hss' \subseteq sKs'\subseteq H$. If $e\in E(H)$ then $ss'ess' = ess' = ss'e \in H$ and since $H$ is unitary in $S$ then $ss'\in H$ from which we deduce that $s's\in K$. Therefore $(sK)\omegah$ is an $\omegah-$coset of $K$ in $S$. Now suppose that $t\in S_{(sK)\omegah}$. Then $((ts)K)\omegah = (sK)\omegah$ and so $ts\pi_K s$. Therefore $s'ts\in K$ and so $ss'tss'\in H$ and since $H$ is unitary in $S$ we deduce that $t\in H$. Conversely if $t\in H$ then $s'ts\in K$ and so $s\pi_K ts$ or in other words $((ts)K)\omegah = (sK)\omegah$ and $t\in S_{(sK)\omegah}$. Hence $H = S_{(sK)\omegah}$. Define $\phi:S/H\to S/K$ by $\phi((tH)\omegah) = (tsK)\omegah$ and notice that $\phi$ is a well-defined morphism. To see this note that there exists $t'\in W(t)$ with $t't\in H$. It follows that $s'(t't)s \in K$ and since $s't'\in W(ts)$ then $((ts)K)\omegah$ is an $\omegah-$coset. If $(tH)\omegah=(rH)\omegah$ then there exists $r'\in W(r)$ such that $r'r,r't\in H=S_{(sK)\omegah}$ and so $(rsK)\omegah = (tsK)\omegah$. Finally, as $H = S_{(sK)\omegah}$ then $\phi$ is injective and as $S/K$ is transitive then $\phi$ is onto and so an isomorphism as required.
\end{proof}

In fact we can go a bit further

\begin{theorem}
Let $H$ and $K$ be closed $E-$dense subsemigroups of an $E-$dense semigroup $S$. Then $H$ and $K$ are conjugate if and only if there is exist $s \in S, s'\in W(s)$
$$
(s'Hs)\omegah = K {\rm\ and \ } (sKs')\omegah = H.
$$

Moreover, any such element $s$ necessarily satisfies $ss'\in H, s's\in K$.
\end{theorem}

\begin{proof}
From Theorem~\ref{clifford_conjugate}, if $H$ and $K$ are conjugate, then there exists $s \in S, s'\in W(s)$ such that
$$
s'Hs \subseteq K {\rm\ and \ } sKs'\subseteq H.
$$
Now it is clear that $(s'Hs)\omegah \subseteq K$ so let $k \in K, k'\in W(k)\cap K$ and let $l = skk'ks' \in H$. Now put $m = s'ls=s'skk'ks's\in s'Hs$ and notice that $m \leh k$ and so $k \in (s'Hs)\omegah$ as required.
\end{proof}

\medskip

Notice that if $ss' \in H$ then $s'Hs$ is an $E-$dense subsemigroup of $H$. To see this note that it is clearly a subsemigroup and that $s'h'ss's\in W(s'hs)\cap s'Hs$ for any $h'\in W(h)\cap H$. In particular from Theorem~\ref{S_x-closed-theorem} we immediately deduce

\begin{proposition}
Let $S$ be an $E-$dense semigroup with a semilattice of idempotents $E$ and let $X$ be an $E-$dense $S-$act. Let $s \in S$ and $x \in D_s$. Then $sS_xs'$ is an $E-$dense subsemigroup of $S$ for any $s'\in W(s)\cap D^{sx}$.
\end{proposition}

\begin{theorem}\label{stabilizer_conjugate}
Let $S$ be an $E-$dense semigroup with semilattice of idempotents $E$ and let $X$ be an $E-$dense $S-$act. Let $s \in S$ and $x \in D_s$. Then $S_x$ and $S_{sx}$ are conjugate.
\end{theorem}
\begin{proof}
Since $Sx = Ssx$ the result follows from Theorem~\ref{orbit-stabilizer-theorem}. In fact, we have that $(sS_xs')\omegah = S_{sx}$ for any $s'\in W(s)\cap D^{sx}$.
\end{proof}

\medskip

If $H$ is a closed $E-$dense subsemigroup of an $E-$dense semigroup $S$ with semilattice of idempotents $E$, then we say that $H$ is {\em self-conjugate} if $H$ is only conjugate to itself.

\begin{proposition}
Let $H$ be a closed $E-$dense subsemigroup of an $E-$dense semigroup $S$ with semilattice of idempotents $E$. Then $H$ is self-conjugate if and only if for all $s\in S$ and all $s'\in W(s)$ such that $s's\in H$ then $sHs'\subseteq H$.
\end{proposition}
\begin{proof}
If $s\in S$ and $s'\in W(s)$ with $s's\in H$ then by Lemma~\ref{conjugate-lemma}, $sHs'\subseteq S_{(sH)\omegah}$. By Theorem~\ref{orbit-stabilizer-theorem}, $S_{(sH)\omegah}$ is conjugate to $H$ and so since $H$ is self-conjugate, $sHs'\subseteq H$.

\smallskip

Conversely, suppose that $K$ is a closed $E-$dense subsemigroup of $S$ and that $K$ is conjugate to $H$. Then by Theorem~\ref{s-h-theorem} there is an $\omegah-$coset, $(sH)\omegah$ say, such that $S_{(sH)\omegah} = K$. Hence there exists $s'\in W(s)$ such that $s's\in H$. If $k\in K$ then $s\pi_H ks$ and so $s'ks\in H$ and in addition, since $s's\in H$ and $sHs'\subseteq H$ then $ss'=s(s's)s' \in H$. Hence $s(s'ks)s' \in sHs'\subseteq H$ and so $K\subseteq H$ as $H$ is unitary in $S$. Since $ss' = (ss's)s'$ and since $ss's\in W(s')$ then $s'H(ss's)\subseteq H$ and so $s'Hs\subseteq H$ as $H$ is unitary. Consequently, for all $h\in H, s\pi_Hhs$ and so $H\subseteq S_{(sH)\omegah}=K$.
\end{proof}

An alternative characterisation of self-conjugacy is given by

\begin{proposition}\label{self-conjugate-proposition}
Let $H$ be a closed $E-$dense subsemigroup of an $E-$dense semigroup $S$ with semilattice of idempotents $E$. Then $H$ is self-conjugate if and only if for all $s,t\in S, st\in H$ implies $ts\in H$.
\end{proposition}
\begin{proof}
If $st\in H$ then there exists $t'\in W(t), s'\in W(s)$ such that $t's'\in W(st)$ and $t's'st\in H$. Hence $(t't)(t's'st) \in H$ and so $t't\in H$. Since $tHt'\subseteq H$ then $tt'\in H$. But then $ts(tt') = tstt'tt' = t(stt't)t' \in tHt'\subseteq H$ and so $ts\in H$ as $H$ is unitary in $S$.

\smallskip

Conversely, suppose that $s\in S, s'\in W(s)$ with $s's\in H$ and let $h\in H$. Then $s'(shs's) = (s's)h(s's) \in H$ and so $shs' = (shs's)s'\in H$ and $H$ is self-conjugate.
\end{proof}

If $H$ is self-conjugate then $S/H$ has a richer structure. First notice that

\begin{lemma}
Let $H$ be a self-conjugate closed $E-$dense subsemigroup of an $E-$dense semigroup $S$ with semilattice of idempotents $E$. Then $D_H$ is a closed $E-$dense subsemigroup of $S$.
\end{lemma}
\begin{proof}
Let $s,t\in D_H$ and $s'\in W(s), t'\in W(t)$ with $s's,t't\in H$. By Proposition~\ref{self-conjugate-proposition} $tt'\in H$ and since $tt't\in W(t')$ then $t'Htt't\subseteq H$. Then $t's'\in W(st)$ and $t's'st=t's'stt't\in t'Htt't\subseteq H$ so that $st\in D_H$. Further, as $ss'\in H$ and as $ss's\in W(s')$ then $s'\in D_H$. Hence $D_H$ is an $E-$dense subsemigroup of $S$. Now suppose that $s\leh r$ with $r\in S$ so that there exist $e,f\in E$ such that $s = re=fr$. Hence there exists $f'\in W(f), r'\in W(r)$ such that $s'=r'f'$ Consequently, $r'rs's = r'rr'f'fr = r'f'fr = s's\in H$ and so $r'r\in H$ as $H$ is unitary in $S$. Hence $D_H$ is a closed $E-$dense subsemigroup of $S$.
\end{proof}

As a consequence we can deduce the following interesting result.

\begin{theorem}
Let $H$ be a self-conjugate closed $E-$dense subsemigroup of an $E-$dense semigroup $S$ with semilattice of idempotents $E$. Then $S/H$ is a group under the multiplication
$$
((sH)\omegah)((tH)\omegah) = ((st)H)\omegah.
$$
\end{theorem}
\begin{proof}
The multiplication given is well defined as if $(s_1H)\omegah=(s_2H)\omegah$ and $(t_1H)\omegah=(t_2H)\omegah$ with $s_1, s_2, t_1,t_2\in D_H$ then there exist $s_1'\in W(s_1), t_1'\in W(t_1)$ such that $s_1's_2, t_1't_2\in H$ and so by Proposition~\ref{self-conjugate-proposition} we have $t_2t_1'\in H$ and $s_1's_2t_2t_1'\in H$ and so $t_1's_1's_2t_2\in H$. Hence $((s_1t_1)H)\omegah = ((s_2t_2)H)\omegah$ as required. Multiplication is clearly associative and it is easy to see that $H\omegah$ is the identity. It is also clear that $(s'H)\omegah \in W((sH)\omegah)$ and so $S/H$ is an $E-$dense monoid. Let $(sH)\omegah\in E(S/H)$ so that $s's^2\in H$. Since $s's\in H$ and $H$ is unitary in $S$ then $s\in H$ and so $|E(S/H)|=1$. Hence by Lemma~\ref{e-dense-group-two-lemma}, $S/H$ is a group.
\end{proof}

In particular, if $H$ is self-conjugate and $D_H=S$ then $\pi_H$ is a group congruence on $S$.

\begin{proposition}
Let $H$ be a self-conjugate closed $E-$dense subsemigroup of an $E-$dense semigroup $S$ with semilattice of idempotents $E$. Then for each $s\in D_H$ $\rho_s:S/H\to S/H$ given by $\rho_s(X) = (sX)\omegah$ is a bijection. The map $\rho : D_H\to \hbox{\rm Sym}(S/H)$ given by $\rho(s) = \rho_s$ is a semigroup homomorphism with $\ker(\rho) = \pi_H$.
\end{proposition}
\begin{proof}
If $s,t \in D_H$ then it is clear that $\rho_{st} = \rho_s\rho_t$ and so $\rho$ is a homomorphism. Let $X\in S/H$ so that $X = (tH)\omegah$ for some $t\in D_H$. If $s\in D_H$ then $st\in D_H$ and so $\rho_s(X) = (sX)\omegah = ((st)H)\omegah\in S/H$. In addition there exists $s'\in D_H\cap W(s)$ and so $s's\in D_H$ and $\rho_{s's}(X) = ((s's)X)\omegah = ((s'st)H)\omegah = (tH)\omegah = X$. In a similar way $\rho_{ss'}(X) = ((ss')X)\omegah = X$ and so $\rho_{s'}$ is the inverse of $\rho_s$.

If $(s,t)\in\ker(\rho)$ then in particular $\rho_s(H\omegah)=\rho_t(H\omegah)$ and so $(s,t)\in\pi_H$. Conversely if $(s,t)\in\pi_H$ then $(sH)\omegah=(tH)\omegah$ and so since $S/H$ is a group then for any $r\in D_H$, $(srH)\omegah = (trH)\omegah$ or in other words $(sX)\omegah=(tX)\omegah$ for any $X\in S/H$. Therefore $(s,t)\in\ker(\rho)$.
\end{proof}

\section{Semigroup acts and the discrete log problem}
Many modern cryptographic applications make implicit use of the inherent difficulty of solving the discrete log problem. In this section we consider the problem from an abstract perspective focussing on the (total) action of semigroups on sets (see~\cite{maze-07} for more details of this approach).

\smallskip

Let $S$ be a semigroup and $X$ a (total) left $S-$act. Suppose also that the action on $X$ is {\em cancellative} in the sense that for all $s\in S$ and all $x,y \in X$ if $sx=sy$ then $x=y$. For each $s\in S$ we shall call the pair $(X,s)$ an {\em $S-$cryptosystem} with encryption (function) $x\mapsto sx$. We refer to $x$ as the {\em plaintext}, $sx$ as the {\em ciphertext} and $s$ as the {\em cipher key}. Our problem is then to find a {\em decrypt key} $t$ such that $(ts)x = x$.

If for $s\in S, x,y\in X$ we have $y=sx$ then we shall refer to $s$ as the {\em discrete log} of $y$ {\em to the base} $x$. The {\em discrete log problem} is then to compute the value of $s$ given $sx$ and $x$. In general of course, the discrete log of $sx$ may not be unique.

\medskip

As an example, let $S=U_{p-1}$ be the group of units of the ring $\z_{p-1}$ and $X=U_p$ the group of units of $\z_p$ with $p$ a prime. For $n\in S, x \in X$ define $n\cdot x = x^n\mod p$. By Fermat's little theorem,  if $x$ is a unit modulo $p$, then $x^{p-1}\equiv1\mod p$ and since $n$ is coprime to $p-1$ then there is a positive integer $m$ such that $mn\equiv1\mod p-1$ and hence $x^{mn}\equiv x\mod p$. Consequently $m$ is a decrypt key for $n$. The usefulness of this system lie in the fact that we know of no efficient, non-quantum algorithms, to solve this particular discrete log problem.

\smallskip

More generally, we can let $X$ be a finite group of order $r$ and let $S$ be the group of units of the ring $\z_r$. Then the action $S\times X\to X$ given by $(n,x)\mapsto x^n$ is the basis of an $S-$cryptosystem, in which the inverse of any key $n\in S$ can easily be computed using the Euclidean algorithm. The case when $r=pq$ with $p$ and $q$ being distinct primes, forms the basis of the RSA public-key encryption system.

\medskip

There are in fact a number of well-know algorithms or protocols for public key encryption which depend on the difficulty of solving the discrete log problem. For example

\begin{example}
Massey-Omura
\end{example}
Let $S$ be a commutative semigroup that acts on a set $X$ and suppose that for each $s\in S$ there is an {\em inverse element} $s^{-1}$ with the property that $s^{-1}sx=x$ for all $x\in X$. Suppose now that Alice wants to send Bob a secure message $x$. She chooses a secret random element of the semigroup $s$, say and sends Bob the value $sx$. Bob also chooses a secret random element of the semigroup $t$, say and sends Alice the value $t(sx)$. Alice then computes $tx = (s^{-1}s)(tx) = s^{-1}(t(sx))$ and sends this to Bob. Bob then computes $x = t^{-1}(tx)$ as required.

\smallskip

We can in fact remove the need for $S$ to be commutative if we assume that $X$ is an $(S,S)-$biact. In this case, Alice sends Bob the value $sx$ and Bob sends Alice the value $(sx)t = s(xt)$. Alice then computes $xt = (s^{-1}s)(xt) = s^{-1}(s(xt))$ and Bob then proceeds as normal.

\smallskip

The beauty of such a scheme is that the values of $s$ and $t$ are chosen at random, do not need to be exchanged in advance and do not need to be re-used.

\begin{example}
Generalised ElGamal encryption.
\end{example}
In this system, we again assume that $S$ is a (not necessarily commutative) semigroup that acts on a set $X$ and that a shared secret key, $s\in S$, has previously (or concurrently) been exchanged. Alice chooses a secret random value $c\in S$, while Bob chooses a secret random value $d\in S$ and publishes $sd$ as his public key. Alice then sends the pair of values $((c(sd))x,cs)$ to Bob, who computes $(cs)d = c(sd)$ and hence $(c(sd))^{-1}$ and so recovers $x$. Again the values $c$ and $d$ do not have to be re-used.

\bigskip

It is clear that if $S$ is a group, the {\em inverse} element $s^{-1}$ will always exist, namely the group inverse. For semigroups in general however this may not always be the case. We require that the stabilizers $S_x$ be left dense in $S$ in order to guarantee that the inverse key will exist for all $s\in S$.

\begin{proposition}
Let $S$ be a semigroup and $X$ an $S-$act. The following are equivalent;
\begin{enumerate}
\item for all $x \in X$, $S_x$ is left dense in $S$,
\item for all $x \in X$, $Sx$ is a transitive $S-$act and $x\in Sx$,
\item every locally cyclic $S-$subact of $X$ is transitive and for all $x \in X$, $x\in Sx$.
\end{enumerate}
\end{proposition}
\begin{proof}
$(1)\implies(2)$. For all $s,r\in S$ there exists $t\in S$ such that $tsx= x$ and so $(rt)sx = rx$. Hence $Sx$ is transitive and $x\in Sx$.

$(2)\implies(3)$. Let $Y$ be a locally cyclic $S-$subact of $X$ and let $x,y\in Y$. Then there exists $z\in Y$ such that $x,y\in Sz$ and so since $Sz$ is transitive then there exists $s\in S$ such that $y=sx$ as required.

$(3)\implies(1)$. Let $x\in X$ so that by assumption $x\in Sx$. It is clear that $Sx$ is locally cyclic and so transitive. Consequently, for all $s\in S$ there exists $t\in S$ such that $t(sx) = x$ and $S_x$ is left dense in $S$.
\end{proof}

\smallskip

We also require that $X$ is a cancellative $S-$act. The following result is straightforward to prove, but note that we only require $S$ to be $E-$dense in order to justify $(4)\implies(1)$.

\begin{lemma}\label{e-dense-reflexive-lemma}
Let $S$ be an $E-$dense semigroup and $X$ an $S-$act. The following are equivalent
\begin{enumerate}
\item $X$ is cancellative,
\item for all $x\in X$, $E\subseteq S_x$,
\item for all $x\in X$, $E\omegah\subseteq S_x$,
\item for all $s\in S, s'\in L(s), x\in X$ then $s's\in S_x$.
\end{enumerate}
\end{lemma}

Notice from property (4) that if $X$ is a cancellative $S-$act then for all $x\in X, S_x$ is left dense in $S$. So if $S$ is an $E-$dense semigroup then all cancellative cyclic acts are automatically transitive. So the question arises as to how we can construct a cancellative $S-$act over an $E-$dense semigroup. We do know the structure of $E-$dense transitive acts over $E-$dense semigroups and these are automatically cancellative. In fact it is then clear that if $S$ is $E-$dense, then a total $S-$act $X$ is cancellative if and only if it is an $E-$dense $S-$act in which for each $s\in S, D_s=X$.

\medskip

Let $S$ be an $E-$dense semigroup, let $(X,s)$ be an $S-$cryptosystem and let $s',s''\in L(s)$. Then for any $x \in X$ we see that
$$
s'x = (s''s)(s'x) = s''(ss'x) = s''x.
$$

As with $E-$dense $S-$acts we have
\begin{lemma}
Let $S$ be an $E-$dense semigroup and let $X$ be a cancellative $S-$act. Then for all $x\in X, S_x$ is $\omegah-$closed.
\end{lemma}

\medskip

If $K(s,x) = \{t\in S|ts\in S_x\}$, {\em the decrypt key space}, then we know that $W(s)\subseteq L(s)\subseteq K(s,x)$.

\begin{theorem}\label{ksx-theorem}
Let $S$ be an $E-$dense semigroup, let $(X,s)$ be an $S-$cryptosystem and let $x \in X$. Then

\begin{enumerate}
\item $K(s,x)$ is $\omegam-$closed,
\item $\left(S_xW(s)S_{sx}\right)\omegam\subseteq K(s,x)$,
\item If $E$ is a band then $\left(S_xW(s)S_{sx}\right)\omegah = K(s,x)$,
\item If $S$ is an inverse semigroup then $K(s,x) = \left(S_xs^{-1}\right)\omegah$.
\end{enumerate}
\end{theorem}
\begin{proof}
Let $S, s$ and $x$ be as in the statement of the theorem.
\begin{enumerate}
\item If $t\in K(s,x)$ and if $t\lem r$ then there exist $a,b \in S$ such that $t=ar=rb, at=tb = t$. Hence if $b'\in W(b)$ then $rsx = r(bb'sx) = tb'sx = tbb'sx = tsx = x$ and so $r\in K(s,x)$.
\item Let $t\in S_x$, $s'\in W(s)$ and let $r\in S_{sx}$. Then $(ts'r)(sx) = ts'sx = tx = x$ and so $ts'r\in K(s,x)$ and the result then follows by part (1).
\item Let $t \in K(s,x)$ and notice that for any $s'\in W(s)$ and any $t'\in W(t)$ it follows that $tss't't \in S_xW(s)S_{sx}$. But $ss't't\in E$ since $E$ is a band and $tss't'\in E$ since $S$ is weakly self-conjugate. Hence $tss't't\leh t$ and so $K(s,x)\subseteq (S_xW(s)S_{sx})\omegah$.
\item If $t\in S_x$ then $ts^{-1} = ts^{-1}ss^{-1}\in S_xL(s)S_{sx}$ and so $(S_xs^{-1})\omegah\subseteq K(s,x)$. Conversely, let $t\in K(s,x)$ so that $tsx = x$. Then $tss^{-1}\in S_xs^{-1}$ and since $tss^{-1} \leh t$ the result follows.
\end{enumerate}
\end{proof}
In particular, if $S$ is a group then $K(s,x) = S_xs^{-1}$ and so $|K(s,x)| = |S_x|$. A group $S$ is said to act {\em freely} on a set $X$ if for all $x \in X, S_x = \{1\}$. Clearly in this case, for each key $s$ there is then a unique decrypt key $s^{-1}$. However if the action is not free then there will be more than one decrypt key for at least one $s\in S$. For the classic discrete log cipher $U_{p-1}\times U_p\to U_p$, $(n,x)\mapsto x^n$, the action is indeed a free action. Notice also that for any $E-$dense semigroup $S$ and for all $x\in X$, $E\omegah\subseteq S_x$. As with $E-$dense acts, we shall say that a cancellative $S-$act $X$ is {\em locally free} if for all $x \in X, S_x = E\omegah$. This is a different definition from the usual concept of freeness in $S-$acts (see~\cite{kilp-00}).

\begin{example}
Let $S$ be an $E-$dense semigroup with a band of idempotents $E$ and let $I$ be a left ideal of $S$. Then $I$ is a locally free $S-$act.
\end{example}
To see this suppose that $s\in S, x\in I$ and that $sx = x$. Then for $x'\in W(x), s'\in W(s)$ it follows that $sxx's's = xx's's\in E$ since $E$ is a band. However, $sxx's's = s(xx's's) = (sxx's')s$ and since $xx's's, sxx's'\in E$ it follows that $s\geh sxx's's$ so that $S_x\subseteq E\omegah$.

\medskip

From the point of view of decryption, ideally we need a group acting freely on a set. However, the point of the discrete log problem is not that it is impossible to solve, but rather that it is hard to solve. Perhaps if finding one needle in a haystack is hard, then finding two, or at least a relatively small number, is equally hard. Having said that, we probably wish to minimise the size of $K(s,x)$ and so if $S$ is an $E-$dense semigroup then we may wish to consider those semigroups for which $E\omegah = E$, which in the case of those $E-$dense semigroups with a band of idempotents is, by Lemma~\ref{e-dense-e-unitary-lemma}, an $E-$unitary semigroup. We shall refer to such semigroups as {\em $E-$unitary dense} semigroups. Notice that in this case, if $X$ is locally free, $K(s,x) = \{t|ts\in S_x\} = \{t|ts\in E\} = L(s)$. Notice also that by Lemma~\ref{e-dense-group-lemma}, if $|L(s)|=1$ then $S$ is a group.

\begin{proposition}
Suppose that $S$ is an $E-$unitary dense semigroup with a semilattice of idempotents $E$ and suppose that $X$ is a locally free cancellative $S-$act. Then for all $s\in S, x\in X, K(s,x) = (W(s))\omegah$.
\end{proposition}
\begin{proof}
By Theorem~\ref{ksx-theorem}, $K(s,x) = (EW(s)E)\omegah$ and by Lemma~\ref{weak-inverses-lemma}, $EW(s)E\subseteq W(s)$. Hence $K(s,x)\subseteq (W(s))\omegah$. But if $s'\leh t$ for some $s'\in W(s), t\in S$ then there exist $e,f \in E$ such that $s' = et = ft$. Consequently $fts = s's\in E$ and so $ts\in E$ as $S$ is $E-$unitary. Hence $(W(s))\omegah\subseteq L(s) = K(x,s)$ and the result follows.
\end{proof}

\medskip

There have been many results concerning the structure of $E-$unitary dense semigroups based on the celebrated results of McAlister (\cite{mcalister-74a},~\cite{mcalister-74b}) and we present here a version of the one first given in~\cite{almeida-92}. First notice that if  $S$ is a semigroup and if ${}^1S$ is the monoid obtained from $S$ by adjoining an identity element $1$ (regardless of whether $S$ already has an identity), then $S$ is an $E-$unitary dense semigroup if and only if ${}^1S$ is an ${}^1E-$unitary dense monoid. This observation allows us to present the construction for $E-$unitary dense monoids, without much loss of generality. We use, for the most part, the terminology of~\cite{fountain-04}.

Let $C$ be a small category, considered as an algebraic object, with a set of objects, $\obj C$ and a disjoint collection of sets, $\mor(u,v)$ of morphisms, for each pair of objects $u,v\in\obj C$. The collection of all morphisms of $C$ is denoted by $\mor~C$, for each object $u\in\obj C$ the identity morphism is denoted by $0_u$ and composition of morphisms, denoted by $p+q$ for $p,q\in\mor~C$, is considered as a partial operation on $\mor~C$. Notice that, despite the notation, we do not assume that $+$ is commutative. For each object $u\in\obj C$ the set $\mor(u,u)$ is a monoid under composition and is called the {\em local monoid} of $C$ at $u$. We shall say that $C$ is {\em locally idempotent} if each local monoid $\mor(u,u)$ is a band, and that $C$ is {\em strongly connected} if for every $u,v\in\obj C, \mor(u,v)\ne\emptyset$.

Let $G$ be a group. An {\em action} of the group $G$ on a category $C$, is given by a group action on $\obj C$ and $\mor~C$ such that
\begin{enumerate}
\item if $p\in\mor(u,v)$ then $gp\in\mor(gu,gv)$,
\item $g(p+q)=gp+gq$ for all $g\in G, p,q\in\mor~C$, (whenever both sides are defined),
\item$g0_u=0_{gu}$ for all $g\in G, u\in\obj C$.
\end{enumerate}
The action is said to be {\em transitive} if for all objects $u,v\in\obj C$ there exists $g\in G, gu=v$, and {\em free} if the action on the objects if a free action (i.e. $S_u = \{1\}$ for all $u\in\obj C$). Notice that if the action is both transitive and free then for each pair $u,v\in\obj C$ there exists a unique $g\in G$ with $gu=v$.

Now suppose that $C$ is a strongly connected, locally idempotent category and that the group $G$ acts transitively and freely on $C$. Let $u\in\obj C$ and let
$$
C_u = \{(p,g)|g\in G, p\in\mor(u,gu)\}.
$$
Then $C_u$ is a monoid with multiplication defined by
$$
(p,g)(q,h) = (p+gq,gh).
$$

\begin{theorem}[{\cite[Proposition 3.2, Theorem 3.4]{almeida-92}}]
Let $S$ be a monoid with band of idempotents $E$. Then $S$ is $E-$unitary dense if and only if there exists a strongly connected, locally idempotent category $C$ and a group $G$ that acts transitively and freely on $C$ and $S$ is isomorphic to $C_u$ for some (any) $u\in\obj C$.
\end{theorem}
Notice that the idempotents of $S$ correspond to the elements of the form $(p,1)$. Also, as $K(s,x) = L(s)$, we see that $K((p,g),x) = \{(q,g^{-1})\in S\}$ and so $|K((p,g),x)| = |\mor(u,g^{-1}u)|$. Consequently we see by Lemma~\ref{e-dense-group-lemma} that $S$ is a group if and only if for all $g\in G, |\mor(u,gu)|=1$. In fact we see from Lemma~\ref{e-dense-group-two-lemma} that in order for $G$ not to be a group we require $|E|>1$ (this would not be true if $S$ is not a monoid).

\medskip

Define the {\em support} of the category $C$ to be the underlying graph of $C$. Now consider the following category. Let $\obj C=S$ and for $u,v\in\obj C$ define $\mor(u,v) = \{(u,s,v)|s\in S, v=su\}$. This is called the {\em derived category} of the monoid $S$. The support of $C$ is often called the {\em left Cayley graph} of $S$.

\smallskip

For a specific example of the above construction of an $E-$unitary dense monoid, let $G$ be a group and let $C$ be the derived category of $G$ with the induced action of $G$ on $C$. That is to say $g(u,s,v) = (gu,gsg^{-1},gv)$. Then $C_u$ is an $E-$unitary dense monoid, $C$ is a locally idempotent category on which $G$ acts transitively and freely and $C_u\cong C_1\cong G$. Notice that in this case every morphism in $C$ is an isomorphism and so $C$ is a {\em groupoid}.
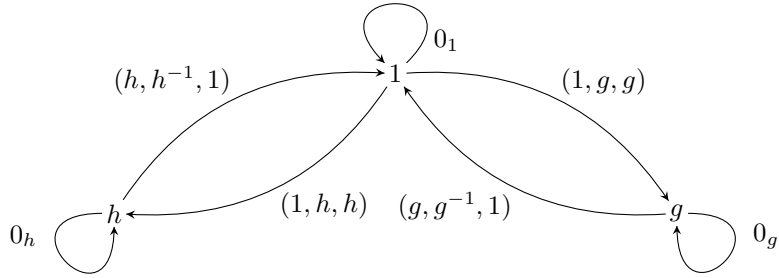
\begin{center}
\begin{tikzpicture}[scale=1.85,inner sep=0.5mm, decoration={markings,mark=at position 0.999 with {\arrow{stealth}}}]
	\node (a) at (2,1) {$1$};
	\node (b) at (4,0) {$g$};
	\node (c) at (0,0) {$h$};
	\path (a) edge [-,postaction={decorate},bend left] node[midway,above right] {$(1,g,g)$} (b)
		(b) edge [-,postaction={decorate},bend left] node[midway,below left] {$(g,g^{-1},1)$} (a)
		(a) edge [loop above,-,in=135,out=45,looseness=10,min distance=8mm,postaction={decorate}] node[at start, right,xshift=3mm,yshift=3mm] {$0_1$} (a)
		(b) edge [loop above,-,in=-90,out=0,looseness=10,min distance=8mm,postaction={decorate}] node[at start, right,xshift=8mm,yshift=-3mm] {$0_g$} (b);
	\path (a) edge [-,postaction={decorate},bend left] node[midway,below right] {$(1,h,h)$} (c)
		(c) edge [-,postaction={decorate},bend left] node[midway,above left] {$(h,h^{-1},1)$} (a)
		(c) edge [loop above,-,in=-90,out=180,looseness=10,min distance=8mm,postaction={decorate}] node[at start, left,xshift=-8mm,yshift=-3mm] {$0_h$} (c);
\end{tikzpicture}
\end{center}

If we wish to work with $E-$unitary dense semigroups rather than monoids, we can simply remove the need for an identity element in $\mor(u,u)$ (see~\cite{almeida-92} for more details).

\smallskip

A simple modification of the previous example can provide us with an $E-$unitary dense semigroup that is not a group. Let $G$ be a group and in the derived category of $G$, replace $\mor(u,u)=\{0_u\}$ with the 2-element band $\{0_u,e_u\}$. Now extend the composition of maps so that we form a category. In other words for each $u\in\obj C, g\in G$ add in the maps $e_u+(u,g,gu), (u,g,gu)+e_{gu}, e_u+(u,g,gu)+e_{gu}$. In addition we can extend the action of $G$ accordingly so that $ge_u = e_{gu}$.

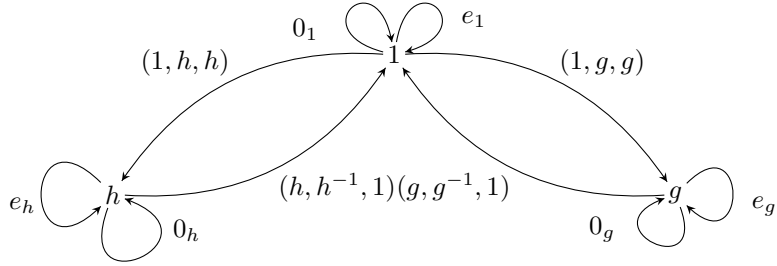
\begin{center}
\begin{tikzpicture}[scale=1.85,inner sep=0.5mm, decoration={markings,mark=at position 0.999 with {\arrow{stealth}}}]
	\node (a) at (2,1) {$1$};
	\node (b) at (4,0) {$g$};
	\node (c) at (0,0) {$h$};
	\path (a) edge [-,postaction={decorate},bend left] node[midway,above right] {$(1,g,g)$} (b)
		(b) edge [-,postaction={decorate},bend left] node[midway,below left] {$(g,g^{-1},1)$} (a)
		(a) edge [loop above,-,in=100,out=165,looseness=10,min distance=6mm,postaction={decorate}] node[at start, left,xshift=-8mm,yshift=3mm] {$0_1$} (a)
		(a) edge [loop above,-,out=80,in=15,looseness=10,min distance=6mm,postaction={decorate}] node[at start, right,xshift=8mm,yshift=3mm] {$e_1$} (a)
		(b) edge [loop above,-,in=-160,out=-70,looseness=10,min distance=6mm,postaction={decorate}] node[at start, left,xshift=-8mm,yshift=-3mm] {$0_g$} (b)
		(b) edge [loop above,-,out=45,in=-45,looseness=10,min distance=6mm,postaction={decorate}] node[at start, right,xshift=8mm,yshift=-3mm] {$e_g$} (b);
	\path (a) edge [-,postaction={decorate},bend right] node[midway,above left] {$(1,h,h)$} (c)
		(c) edge [-,postaction={decorate},bend right] node[midway,below right] {$(h,h^{-1},1)$} (a)
		(c) edge [loop above,-,in=-20,out=-110,looseness=10,min distance=8mm,postaction={decorate}] node[at start, right,xshift=8mm,yshift=-3mm] {$0_h$} (c)
		(c) edge [loop above,-,out=135,in=-135,looseness=10,min distance=8mm,postaction={decorate}] node[at start, left,xshift=-8mm,yshift=-3mm] {$e_h$} (c);
\end{tikzpicture}
\end{center}

Notice then that $(u,g,gu)+e_{gu}+(gu,g^{-1},u)\in\mor(u,u)$ and so must be either equal to $0_u$ or $e_u$. If it were equal to $0_u$ then we can add $(gu,g^{-1},u)$ to the left and $(u,g,gu)$ to the right to deduce that $e_u = 0_u$ which is obviously a contradiction. Hence $(u,g,gu)+e_{gu}+(gu,g^{-1},u) = e_u$ and so for all $u \in\obj C, g\in G$ we have
$$
(u,g,gu) + e_{gu} = e_u + (u,g,gu).
$$
It then follows that $\mor(u,gu) = \{(u,g,gu), e_u+(u,g,gu)\}$ and $C_1$ is an $E-$unitary dense monoid with 2 idempotents and $|K(s,x)|=2$ for all $s\in S$ and $x\in X$. Notice that we can view this monoid in the following way. Let $G$ be a group and $e$ a symbol not in $G$ and let $eG = \{eg|g\in G\}$ be a set in 1-1 correspondence with $G$. Let $S=G\;\dot\cup\;eG$ and extend the multiplication on $G$ to $S$ by setting $e^2=e$, $eg=ge$ for all $g\in G$ and all other products defined by associativity or the multiplication in $G$. Then $S\cong C_1$ and the isomorphism is given by $g\mapsto ((1,g,g),g), eg\mapsto(e_1+(1,g,g),g)=(e_1,1)((1,g,g),g)$. The element $e$ corresponds to $(e_1,1) \in C_1$.

\medskip

By replacing $\mor(u,u)$ by a band of any given size, we should be able to construct an $E-$unitary dense monoid with any finite number of idempotents.

\bigskip

The above construction gives us a mechanism to build a suitable $E-$unitary dense semigroup $S$. However we need $X$ to be a locally free cancellative $S-$act, so let us revisit the theory of $E-$dense $S-$acts. If $S$ is finite (or at least $E$ is finite) and $E$ is a semilattice, then every $E-$dense act is graded and so by Theorem~\ref{free-transitive-graded-theorem}, $X$ is a locally free $E-$dense $S-$act if and only if $X\cong \dot\bigcup Se_i$ for some idempotents $e_i$, where the action is that given in Example~\ref{preston-wagner-representation}. As previously observed, if $X$ is a cancellative total act then it is automatically reflexive and hence an $E-$dense act. Consequently, if $X$ is locally free then as every idempotent acts on $e_i$, we can deduce that for each $i, e_i=f$, the minimum idempotent in $S$. Conversely if $f$ is the minimum idempotent in $S$ then $Sf\cong S/f\omegah$ is a locally free transitive cancellative total $S-$act. We have therefore shown

\begin{theorem}
Let $S$ be a finite $E-$dense semigroup with semilattice of idempotents $E$, let $s\in S$ and let $f$ be the minimum idempotent in $S$. Then $(X,s)$ is a locally free  $S-$cryptosystem if and only if $X\cong \dot\bigcup Sf$. In addition, if $S$ is $E-$unitary then for each $x\in X, |K(s,x)|=|(W(s))\omegah|$.
\end{theorem}

In the above example where $S=G\;\dot\cup\;eG$, the minimum idempotent is $e$ and $X = eG = Ge$ is a locally free cancellative $S-$act and for each $x\in X, |K(s,x)|=2$. In the classic discrete log cipher, $U_{p-1}$ acts freely on $U_p$ by exponentiation, the minimum idempotent is $1\in U_{p-1}$  and in fact $U_p\cong \dot\bigcup_{|U_p|}U_{p-1}$.

\end{document}